\documentclass{article}
\usepackage{amsmath,amssymb,amsthm,graphicx,epsfig,subfigure,float,url}
\usepackage[colorlinks=true]{hyperref}
\usepackage{pdfsync}

\usepackage{verbatim}
\usepackage[applemac]{inputenc}
 
\usepackage[usenames,dvipsnames]{pstricks}
\usepackage{graphicx}
\usepackage{subfloat}

\usepackage{pst-grad} 
\usepackage{pst-plot} 

\def\R{\textrm{I\kern-0.21emR}}
\def\N{\textrm{I\kern-0.21emN}}
\def\1{1\textnormal{\kern -0.21emI}}

\renewcommand{\geq}{\geqslant}
\renewcommand{\leq}{\leqslant}

\newtheorem{theorem}{Theorem}

\newtheorem{proposition}{Proposition}

\newtheorem{definition}{Definition}
\newtheorem{lemma}{Lemma}
\theoremstyle{definition}
\theoremstyle{definition}\newtheorem{remark}{Remark}

\newcommand{\Hun}{\mathbf{(H_1)}}
\newcommand{\Hdeux}{\mathbf{(H_2)}}

\title{What is the optimal shape of a fin for one dimensional heat conduction?\thanks{The third author was partially supported by ANR Project OPTIFORM.}}

\author{Gilles Marck\footnote{Universit\'e Pierre et Marie Curie (Univ. Paris 6), CNRS UMR 7598, Laboratoire Jacques-Louis Lions, F-75005, Paris, France ({\tt gilles.marck@upmc.fr}).} 
\and Gr\'egoire Nadin\footnote{CNRS, Universit\'e Pierre et Marie Curie (Univ. Paris 6), UMR 7598, Laboratoire Jacques-Louis Lions, F-75005, Paris, France ({\tt gregoire.nadin@upmc.fr}).}
\and Yannick Privat\footnote{CNRS, Universit\'e Pierre et Marie Curie (Univ. Paris 6), UMR 7598, Laboratoire Jacques-Louis Lions, F-75005, Paris, France ({\tt yannick.privat@upmc.fr}).}
}

\date{}

\begin{document}

\maketitle

\begin{abstract}
This article is concerned with the shape of small devices used to control the heat flowing between a solid and a fluid phase, usually called \textsl{fin}. 
The temperature along a fin in stationary regime is modeled by a one-dimensional Sturm-Liouville equation whose coefficients strongly depend on its geometrical features. 
We are interested in the following issue: is there any optimal shape maximizing the heat flux at the inlet of the fin? Two relevant constraints are examined, by imposing 
either its volume or its surface, and analytical nonexistence results are proved for both problems. Furthermore, using specific perturbations, we explicitly compute the optimal values and construct maximizing sequences. We show in particular that the optimal heat flux at the inlet is infinite in the first case and finite in the second one. Finally, we provide several extensions of these results for more general models of heat conduction, as well as several numerical illustrations.
\end{abstract}

\noindent\textbf{Keywords: }heat conduction, calculus of variation, shape optimization, Sturm-Liouville equation, volume constraint, surface constraint. 

\medskip

\noindent\textbf{AMS classification:} 49J15, 49K15, 34E05.

\section{Introduction}\label{secintro}
The increasing need for compact and efficient thermal systems leads to new challenges in the design of heat transfer devices. Across several engineering fields dealing with thermal management issues, a recurrent problematic concerns the optimal shape of several small elements aiming at locally increasing heat transfer. Theses systems, usually called fins, may adopt various shapes and designs, slightly extending the structure subjected to thermal loads.

For instance, fins are widely used to control the temperature of the heat exchangers taking place over the computational processing units (CPU), since thermal overloads can have devastating effects on their physical integrity. Fins are also extensively used in industrial forming processes, mainly to evacuate the heat handled by the molded material, and must be designed to guarantee that the bulk temperature evolves inside a range specified by manufacturing constraints. 

Therefore, efficient fin shapes are required either to improve cooling or warming processes, controlling the local temperature or increasing the heat transfer. Many engineering works focused on modeling the direct problem in order to assess the efficiency of different fin shapes \cite{Bergman}. Nevertheless, very few theoretical results and mathematical proofs do exist (see for instance \cite{Bobaru,Grodzovskii,huang-wuchiu}). The study proposed in the present article tackles this shape optimization problem, taking into account several relevant constraints on the admissible class of designs.

Engineering motivations behind the fin optimization are rooted in a category of physical problems related to efficient transport of a conservative flux. In particular, optimizing the shape of a fin belongs to a larger class of problems aiming at reducing the thermal resistance occurring when the heat flows inside different media. Indeed, the main purpose of a fin is to thermally link a solid material subject to a heat flux with a fluid flow taking place around it. Its role consists in cooling down/warming up the solid phase by transferring heat to/from the fluid phase. 
 
From the optimal design point of view, the main difficulties arise from the multi-physic aspects of this problem, depicted on Fig.~\ref{Fig:fin}. Indeed, the heat flowing from a solid domain through a fin to a fluid part is only transported by conduction inside the structure, as far as it reaches its boundary. Then, according to heat flux conservativeness property, it is fully transferred into the fluid phase. Hence, it is at the same time transported by the fluid motion and conducted between fluid elements, which is referred to as \textsl{conducto-convection phenomenon} (see \cite{Welty}). As a consequence, the heat flowing from the fin to the fluid depends on the flow pattern around boundaries. This multi-physics optimal design problem with a Navier-Stokes/heat coupling system has been numerically investigated for some academic configurations in~\cite{marck}. For the needs of simple and robust modeling, physical/engineering works often only takes into consideration the conduction phenomenon occurring inside the bulk material, considering an average convective coefficient standing for the heat transfer at the solid/fluid interface \cite{Bergman,Welty}. This is the choice we make in the sequel.

In this article, we will consider a fin subject to a steady-state thermal regime and attached to a device at constant temperature. A fin is considered as thermally efficient if it conveys the largest amount of heat, while requiring the smallest volume of solid material. Two reasons motivate this claim: first, a fin is generally made of expensive components, because it requires high conductivity materials such as copper. Secondly, a fin is generally oriented orthogonally to the cross-flow direction, and we could expect that it is designed to generate the smallest possible perturbation in the fluid motion. Indeed, a smaller fin produces a smaller perturbation and requires a smaller additional power to set the fluid into motion.

This feature leads us to deal with the following optimal design problem: maximize the heat crossing the fin, by prescribing either its volume or its surface. From a mathematical point of view, the problems settled within this article write as infinite dimensional optimization problems subject to constraints of several natures: (i) global ones such as volume or surface, (ii) pointwise ones, and (iii) ordinary differential equation ones, since the cost functional depends on the solution of a Sturm-Liouville equation, whose coefficients write as highly nonlinear functions of the unknown.

The idea of minimizing functionals depending on a Sturm-Liouville operator, such as eigenvalues or eigenfunctions, is a long story and goes back at least to M. Krein in \cite{krein}, see also \cite{horvath,huang,mahar-willner,mahar-willner2} and \cite{henrot} for a review on such problems. 
We also mention \cite{henrotPrivat,privat} where two problems close to the ones addressed in this article have been solved in the context of Mathematical Biology (see Remark \ref{rk:comp} for a comparison between these problems). 

\bigskip

The aim of the present paper is to investigate theoretically the issue of maximizing the heat flux under at the inlet of a fin with respect to its shape, under surface or volume constraints. This analysis will be carried out 
using the simplified model (\ref{eq:T}) for the fin, for which we will prove that optimal shapes do not exist, but that one can construct a maximizing sequence in the volume constraint 
case and compute explicitely the optimal value of the heat flux. In the related work \cite{Belinskiy},  
the authors considered a simplified nonstationary model of fin, getting rid of the convective heat transfer from the lateral side of the structure ($\beta=0$ in equation 
(\ref{eq:T}) below), and investigated 
a different criterion (the cooling rate of the fin), for which they completely solved the maximization problem using rearrangement techniques.
We also mention \cite{ChangCleaverChen1982,ChiuChen2002,LiawYeh1994,MarinElliottHeggsInghamLesnicWen2004}, in which more involved models are investigated exclusively in a numerical way.
The novelty of this article comes from not only the awareness of the convective heat transfer, making the optimal design problem harder to tackle but also from the fact that we completely solve it theoretically. This is why the present study can be viewed, in some sense, as a first step in order to understand and derive analytical results for more involved models. 
 
This article is structured as follows: in Section \ref{sec:model1D}, a Sturm-Liouville model of temperature along the fin is derived using physical arguments. Section \ref{sec:Fa} is devoted to introducing the cost functional $F(a)$, standing for the heat flux at the inlet of the fin, and providing continuity properties about it. The two main optimal design problems investigated are introduced in Section \ref{sec:Optim8}: in a nutshell, the first one consists in maximizing $F(a)$ with a volume constraint whereas the second one aims to maximize $F(a)$ with a lateral surface constraint. These problems are solved in Section \ref{sec:mainRes} and proved, respectively in Sections \ref{sec:optvol} and \ref{sec:optS}. More precisely, we show in each case a nonexistence result, provide the optimal value for this problem, as well as a way to construct maximizing sequences. Finally, all these results are extended to a more general setting in Section \ref{sec:ccl}. Several numerical illustrations of the results are also presented.

\section{Modeling of an axisymmetric fin}\label{sec:model_fin}
\subsection{Sturm-Liouville model of the temperature conduction}\label{sec:model1D}

Let us consider an axisymmetric fin $\Omega_{a}$ of length $\ell>0$ and radius $a(x)$ at abscissa $x$, as displayed in Figure~\ref{Fig:fin}, defined in a Cartesian coordinate system by
$$
\Omega_{a}=\{(r\cos\theta,r\sin\theta,x)\mid r\in [0,a(x)), \ \theta\in \mathbb{S}^1, \ x\in (0,\ell)\}.
$$
We will assume in the sequel that 
\begin{itemize}
\item[$\Hun$] $a\in W^{1,\infty}(0,\ell)$ so that the surface element is defined almost everywhere;
\item[$\Hdeux$] there exists $a_{0}>0$ such that $a(x)\geq a_{0}$ for every $x\in [0,\ell]$ so that the fin cannot collapse.
\end{itemize}
Figure \ref{Fig:fin} sums-up the situation and the notations we will use throughout this article. 
\begin{figure}[h!]
	\begin{center}
		\includegraphics[width=0.9\textwidth]{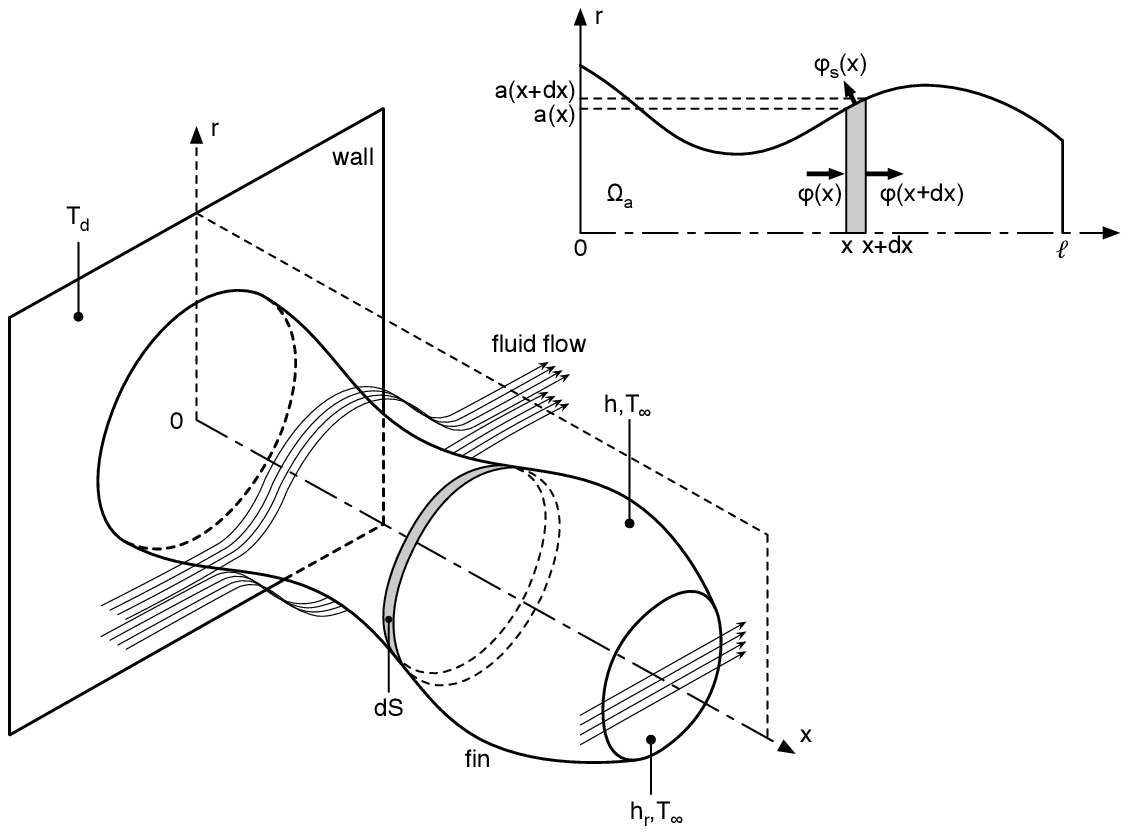}
	\end{center}
	\caption{Scheme of the axisymmetric fin}
	\label{Fig:fin}
\end{figure}

Some physical explanations about the derivation of the temperature model may be found in \cite{Bergman,Welty}. Introduce the convective coefficient $h>0$ modeling the heat transfer between the fin surface and the fluid flow, the convective coefficient $h_r>0$ characterizing the heat transfer over the tip, and $k>0$ the thermal conductivity of the fin.

The inlet of the fin, as well as the fluid surrounding the fin are assumed to be at a constant temperatures, denoted respectively $T_d$ and $T_{\infty}$. Considering processes where the fin aims at cooling a thermal system, \textit{i.e.} where the heat flows from its basis towards the fluid, we will assume that $0<T_{\infty}<T_{d}$.

Then, the temperature $T$ along the fin is solution of the following ordinary differential equation (see ~\cite{Bergman,Welty})
\begin{equation}\label{eq:T}
\begin{array}{ll}
(a^2(x)T'(x))'= \beta a(x)\sqrt{1+a'(x)^2}(T(x)-T_{\infty}) & x\in (0,\ell)\\
T(0)=T_{d} & \\
T'(\ell) = -\beta_r(T(\ell)-T_{\infty}), & 
\end{array}
\end{equation}
where $\beta=2h/k$ is a positive real constant and $\beta_{r}=h_{r}/k$ is a nonnegative real constant.

\begin{remark} \label{rek:model}
Let us comment on this model and the assumptions we made on the physical parameters $k$ and $h$.
\begin{itemize}
\item From a rigorous point of view, the thermal conductivity $k$ depends on the temperature $T$. However, if the operating temperature range where the fin takes place is small enough,  it is reasonable to consider $k$ as a constant parameter. Indeed, knowing that $T_{\infty} \leq T(x) \leq T_d$ as proved in lemma~\ref{lemma1:T}, this hypothesis is valid if the variation of the thermal conductivity between $k(T_{\infty})$ and $k(T_d)$ is small. As a matter of fact, let consider a fin made of copper with working temperatures ranging from $0^{\circ}C$ to $100^{\circ}C$: the order of magnitude of the copper thermal conductivity variation is about $2\%$. This assumption is widely used in the dedicated scientific/engineering literature~\cite{Bergman,LiawYeh1994,MarinElliottHeggsInghamLesnicWen2004,Welty}.
\item The same conclusion holds for the convective coefficient $h$, by assuming that it obeys Newton's law, which means that it does not depend on the difference between $T_{\infty}$ and the fin surface temperature. In the absence of radiative phenomena, this is a standard assumption for several fin models in literature~\cite{Bergman,ChiuChen2002,MarinElliottHeggsInghamLesnicWen2004,Welty}.
\end{itemize}
\end{remark}

Since the function $a$ is assumed to satisfy the assumptions $\Hun$ and $\Hdeux$, System \eqref{eq:T} has a unique solution $T\in H^1(0,\ell)$ by virtue of the Lax-Milgram lemma. The following lemma provides the monotonicity of this solution.
\begin{lemma}\label{lemma1:T}Let $a$ be a function satisfying the assumptions $\Hun$ and $\Hdeux$. Then, the solution $T(\cdot)$ of \eqref{eq:T} is decreasing, and satisfies
$$
T_{\infty}\leq T(x)\leq T_{d}
$$
for every $x\in [0,\ell]$.
\end{lemma}
\begin{proof}
Let us use the standard change of variable for Sturm-Liouville problems (see for example \cite{coxlip}) $y=\int_{0}^x\frac{dt}{a(t)^2}$. Setting $\ell_{1}=\int_{0}^\ell \frac{dt}{a(t)^2}$ and $S(y)=T(x)-T_{\infty}$ for every $x\in [0,\ell]$, it follows from \eqref{eq:T} that 
$S$ is solution of the boundary value problem
\begin{equation}\label{eq:S}
\begin{array}{ll}
S''(y)= \beta \rho(y)S(y) & y\in (0,\ell_{1})\\
S(0)=T_{d}-T_{\infty} & \\
S'(\ell_{1})=-\beta_{r}a(\ell)^2S(\ell_{1}), & 
\end{array}
\end{equation}
where $\rho(y)=a(x)^3\sqrt{1+a'(x)^2}$ for every $x\in [0,\ell]$. Let us now prove that the function $S(\cdot)$ remains positive on $[0,\ell_{1}]$. 
We denote by $y_{0}$ the largest zero of $S(\cdot)$ on $[0,\ell_{1}]$ whenever it exists. One has necessarily $S'(y_0)\neq 0$. Otherwise, the Cauchy-Lipschitz theorem 
would immediately yield the contradiction. In particular, $y_0\neq \ell_1$.
If $S'(y_{0})>0$, then $S$ is positive on $(y_0,\ell_1)$ since $y_0$ is the largest zero of $S$, and thus $S(\cdot)$ is strictly convex and increasing on $(y_{0},\ell_{1})$,
which is incompatible with the boundary 
condition at $\ell_{1}$. In the same way, if $S'(y_{0})<0$, $S(\cdot)$ is concave negative and decreasing on $[y_{0},\ell_{1}]$ and the same conclusion follows. As a result, 
the function $S(\cdot)$ is positive on $[0,\ell_{1}]$, and hence strictly convex according to \eqref{eq:S}. The function $S(\cdot)$ is thus decreasing according to the boundary condition at 
$\ell_{1}$. The conclusion follows.
\end{proof}
\subsection{The cost functional $F(a)$}\label{sec:Fa}
The cost functional $F(a)$ stands for the heat flux at the inlet of the fin. It is defined by
\begin{equation}\label{def:F(a)}
F(a)=-k\pi a(0)^2T'(0),
\end{equation} 
where $T(\cdot)$ denotes the solution of \eqref{eq:T}.
Integrating the main equation of \eqref{eq:T} yields the new (integral) expression 
\begin{equation}\label{def:F(a)2}
F(a)=k\pi \beta \int_{0}^\ell a(x)\sqrt{1+a'(x)^2}(T(x)-T_{\infty})\, dx+k\pi \beta_{r}a(\ell)^2(T(\ell)-T_{\infty}).
\end{equation}

In the forthcoming analysis of the optimal design problems we will investigate, we will need to use continuity properties of the cost functional $F$. Let us write precisely these properties. Define the class of admissible designs
$$
\mathcal{A}_{a_{0},\ell}=\left\{a\in W^{1,\infty}(0,\ell), \ a\geq a_{0}\textrm{ a.e. in }(0,\ell)\right\},
$$
and the product space $\widehat{\mathcal{A}}_{a_{0},\ell}$ defined by
$$
\widehat{\mathcal{A}}_{a_{0},\ell}=\left\{(a,b), \ a\in \mathcal{A}_{a_{0},\ell}\textnormal{ and }b=a\sqrt{1+a'^2}\right\}.
$$
Introduce the functional $\widehat F$ defined on $\widehat{\mathcal{A}}_{a_{0},\ell}$ by
\begin{equation}\label{defhatF1}
\widehat F(a,a\sqrt{1+a'^2})=F(a),
\end{equation}
for every $a\in \mathcal{A}_{a_{0},\ell}$. Here and in the sequel, the notation $\mathcal{M}(0,\ell)$ stands for the space of Radon measures on $(0,\ell)$. 
\begin{definition}
Let $(a_{n},b_{n})_{n\in\N}$ be a sequence of elements of $\widehat{\mathcal{A}}_{a_{0},\ell}$. We will say that $(a_{n},b_{n})_{n\in\N}$ $\tau$-converges to $(a,b)\in C^{0}([0,\ell])\times \mathcal{M}(0,\ell)$ if
\begin{itemize}
\item $(a_{n})_{n\in\N}$ converges to $a$, locally uniformly in $(0,\ell]$;
\item $(b_{n})_{n\in\N}$ converges to $b$ in the sense of measures.
\end{itemize}
\end{definition}
We endow $\widehat{\mathcal{A}}_{a_{0},\ell}$ with the topology inherited from the $\tau$-convergence. One has the following continuity result of the cost functional $F$.
\begin{proposition}\label{prop:FC0}
Let $(a_{n},b_{n})_{n\in\N}$ be a sequence of elements of $\widehat{\mathcal{A}}_{a_{0},\ell}$ which $\tau$-converges to $(a,b)\in C^{0}([0,\ell])\times \mathcal{M}(0,\ell)$. Then, the sequence $(F(a_{n}))_{n\in\N}$ converges to $\widehat F(a,b)$ defined by
\begin{equation}\label{hatF}
\widehat F(a,b)=k\pi \beta \langle b,T-T_{\infty}\rangle_{\mathcal{M}(0,\ell),C^0([0,\ell])}+k\pi \beta_{r}a(\ell)^2(T(\ell)-T_{\infty}),
\end{equation}
where $T=\tilde T+T_{d}$, and $\tilde T$ is the unique solution of the equation written under variational form: find $\tilde T\in H^1(0,\ell)$ satisfying $\tilde T(0)=0$ such that for every test function $\varphi\in H^1(0,\ell)$ satisfying $\varphi(0)=0$, one has
\begin{equation}\label{eq:Tbis}
\int_{0}^\ell a(x)^2\tilde T'(x)\varphi'(x)\, dx +\beta\langle b, (\tilde T+T_{d}-T_{\infty})\varphi\rangle_{\mathcal{M}(0,\ell),C^0([0,\ell])}+\beta_{r}a(\ell)^2(\tilde T(\ell)+T_{d}-T_{\infty})\varphi(\ell)=0.
\end{equation}
\end{proposition}
\begin{remark}
It follows from Proposition \ref{prop:FC0} that the functional $\widehat{F}$ defined (with a slight abuse of notation) by \eqref{hatF} on the closure of $\widehat{\mathcal{A}}_{a_{0},\ell}$ for the topology associated to the $\tau$-convergence is a continuous extension of the function $\widehat F$ defined by \eqref{defhatF1}.
\end{remark}

\begin{proof}[Proof of Proposition \ref{prop:FC0}]
Consider a sequence $(a_{n},b_{n})_{n\in\N}$ of elements of $\widehat{\mathcal{A}}_{a_{0},\ell}$ as in the statement of Proposition \ref{prop:FC0}. Denote by $T_{n}$ the associated solution of \eqref{eq:T}. Let us multiply the main equation of \eqref{eq:T} by $T_{n}-T_{\infty}$ and then integrate by parts. One gets
\begin{eqnarray*}
a_{0}\min\{\beta,a_{0}\}\Vert T_{n}-T_{\infty}\Vert_{H^1(0,\ell)}^2 & \leq &\int_{0}^\ell \left(a_{n}(x)^2T_{n}'(x)^2+\beta b_{n}(x)(T_{n}(x)-T_{\infty})^2\right)\, dx\\
& = & -\beta_{r}a_{n}(\ell)^2(T_{n}(\ell)-T_{\infty})^2-a_{n}(0)^2T_{n}'(0)(T_{d}-T_{\infty}).
\end{eqnarray*}
Integrating Equation \eqref{eq:T} on $(0,\ell)$ yields 
$$
-\beta_{r}a_{n}(\ell)^2(T_{n}(\ell)-T_{\infty})-a_{n}(0)^2T_{n}'(0)=\beta \int_{0}^\ell b_{n}(x)(T_{n}(x)-T_{\infty})\, dx
$$ 
and according to Lemma \ref{lemma1:T}, it follows that
$$
a_{0}\min\{\beta,a_{0}\}\Vert T_{n}-T_{\infty}\Vert_{H^1(0,\ell)}^2\leq \beta_{r}\Vert a_{n}\Vert_{\infty}^2(T_{d}-T_{\infty})^2+ \beta (T_{d}-T_{\infty})^2\langle b_{n},1\rangle_{\mathcal{M}(0,\ell),C^0([0,\ell])} .
$$
Since $(a_{n},b_{n})_{n\in\N}$ $\tau$-converges to $(a,b)$, we deduce from the previous estimate that the sequence $(T_{n})_{n\in\N}$ is uniformly bounded in $H^1(0,\ell)$. Hence, using a Rellich theorem, there exists $T^*\in H^1(0,\ell)$ such that, up to a subsequence, $(T_{n})_{n\in\N}$ converges to $T^*$, weakly in $H^1(0,\ell)$ and strongly in $L^2(0,\ell)$. Introduce for every $n\in\N$, the function $\tilde T_{n}:=T_{n}-T_{d}$. The function $\tilde T_{n}$ is the unique solution of the system whose variational form writes: find $\tilde T_{n}\in H^1(0,\ell)$ such that $\tilde T_{n}(0)=0$ and for every test function $\varphi\in H^1(0,\ell)$ satisfying $\varphi(0)=0$, one has
$$
\int_{0}^\ell \left(a_{n}(x)^2\tilde T_{n}'(x)\varphi'(x) +\beta b_{n}(x)(\tilde T_{n}(x)+T_{d}-T_{\infty})\varphi(x)\right)\, dx=-\beta_{r}a_{n}(\ell)^2(\tilde T_{n}(\ell)+T_{d}-T_{\infty})\varphi(\ell).
$$
The conclusion follows hence easily, passing to the limit into this variational formulation, and noting that
$$
\frac{\widehat{F}(a_{n},b_{n})}{k\pi}=- a_{n}(0)^2T_{n}'(0)= \beta \int_{0}^\ell b_{n}(x)(T_{n}(x)-T_{\infty})\, dx+\beta_{r} a_{n}(\ell)^2(T_{n}(\ell)-T_{\infty}),
$$
since the injection $H^1(0,\ell)\hookrightarrow C^0([0,\ell])$ is compact.
\end{proof}

\section{Optimal design problems for a simplified model of fin}\label{sec:Optim8}
\subsection{The optimal design problems}\label{sec:odp1}
In this section we will introduce and solve the problems modeling the optimal shape of a fin for the Sturm-Liouville model \eqref{eq:T}. We will consider two kinds of constraints: a volume constraint or a lateral surface constraint. Let us define the volume and lateral surface functionals by
$$
\operatorname{vol}(a)=\int_{0}^\ell a(x)^2\, dx\quad \textnormal{ and }\quad \operatorname{surf}(a)=\int_{0}^\ell a(x)\sqrt{1+a'(x)^2}\, dx.
$$
We investigate the problem of maximizing the functional $a\mapsto F(a)$ defined by \eqref{def:F(a)} with a volume or a lateral surface constraint\footnote{Recall that
$$
\textnormal{volume of }\Omega_{a}=\displaystyle \pi \int_{0}^\ell a(x)^2 \, dx,\quad\textrm{and}\quad
\textnormal{lateral surface of }\Omega_{a}=\displaystyle  2\pi \int_{0}^\ell a(x)\sqrt{1+a'(x)^2} \, dx.
$$
}.

\begin{quote}
\noindent{\bf Optimal design problem with volume or lateral surface constraint.}
\textit{Let $a_{0}$ and $\ell$ denote two positive real numbers. Fix $V_{0}>\ell a_{0}^2$ and $S_{0}> \ell a_{0}$. The optimal design problem with volume constraint writes
\begin{equation}\label{defV}
\sup \{F(a),\ a\in \mathcal{V}_{a_{0},\ell,V_{0}}\},
\end{equation}
where
$$
\mathcal{V}_{a_{0},\ell,V_{0}}=\left\{a\in W^{1,\infty}(0,\ell), \ a\geq a_{0}\textnormal{ a.e. in }[0,\ell]\textnormal{ and }\pi \operatorname{vol}(a)\leq \pi V_{0}\right\}
$$
and the optimal design problem with lateral surface constraint writes
\begin{equation}\label{defS}
\sup \{F(a),\ a\in \mathcal{S}_{a_{0},\ell,S_{0}}\},
\end{equation}
where 
$$
\mathcal{S}_{a_{0},\ell,S_{0}}=\left\{a\in W^{1,\infty}(0,\ell), \ a\geq a_{0}\textnormal{ a.e. in }[0,\ell]\textnormal{ and }2\pi \operatorname{surf}(a)\leq 2\pi S_{0}\right\},
$$
}
\end{quote}

\begin{remark}
Physically, the lateral surface constraint can be justified by considering the fluid flowing around the fin. In realistic engineering configurations, the fluid is put into motion by an external equipment, such as a pump or a fan. The pressure and kinetic energies provided by this system are dissipated all through the fluid flow, mainly because of the wall shear stress. In other words, one part of the dissipation is due to the fluid friction against the walls. As a consequence, limiting the wet surface of the fin is a suitable way to reduce the operating cost of the whole thermal system. In addition, limiting the fin surface may also help to reach easier and cheaper shapes to manufacture.
\end{remark}
\begin{remark}\label{rk:comp}
As underlined in the introduction, the optimal design problems settled here look similar to the ones addressed in \cite{henrotPrivat,privat}, devoted to the issue of understanding the nerve fibers shapes by solving an optimization problem. In these works, a Sturm-Liouville operator whose coefficients depended on the shape of the nerve fiber were also introduced and two criteria were considered: a kind of transfer function, and the first eigenvalue of a self-adjoint operator. However, the techniques implemented in the present work are rather different and appear a bit more sophisticated. Indeed, the ideas of the proofs in \cite{henrotPrivat,privat} were all based on the standard change of variable for Sturm-Liouville problems used in Lemma \ref{lemma1:T}, which permitted to consider auxiliary problems and thus get a lower bound of the optimal values. This trick was then used to construct minimizing sequences. Unfortunately, in the present work, we did not manage to adapt such techniques, and the arguments of the proofs rest upon the use of 
particular perturbations that are constructed in the spirit of rearrangement techniques (see Remark \ref{rk:rearrangement} and the proof of Theorem \ref{thpb1D:surf}).
\end{remark}

Before solving these shape optimization problems, let us give some precisions on the topological nature of the classes $\mathcal{V}_{a_{0},\ell,V_{0}}$ and $\mathcal{S}_{a_{0},\ell,S_{0}}$ in $L^\infty(0,\ell)$. 
As it will be highlighted in the proofs of theorems \ref{thpb1D:vol} and \ref{thpb1D:surf}, the classes $\mathcal{V}_{a_{0},\ell,V_{0}}$ and 
$\mathcal{S}_{a_{0},\ell,S_{0}}$ are not closed for the standard strong topology of $W^{1,\infty}(0,\ell)$. The following lemma investigates the $L^{\infty}$-boundedness of the elements of these classes.
\begin{lemma}\label{lemma:10H28}
Let $a_{0}$ and $\ell$ denote two positive real numbers, and let $V_{0}>\ell a_{0}^2$ and $S_{0}> \ell a_{0}$. Then,
\begin{itemize}
\item the class $\mathcal{V}_{a_{0},\ell,V_{0}}$ is not a bounded set of $L^\infty(0,\ell)$,
\item the class $\mathcal{S}_{a_{0},\ell,S_{0}}$ is a bounded set of $L^\infty(0,\ell)$ and for every $a\in \mathcal{S}_{a_{0},\ell,S_{0}}$,
$$
a_{0}\leq a(x)\leq \sqrt{S_{0}^2/\ell^2+4S_{0}}.
$$
\end{itemize}
\end{lemma}
\begin{proof}
First consider the sequence of functions $(a_{n})_{n\geq 1}$ defined by
$$
a_{n}(x)=\left\{\begin{array}{ll}
\sqrt{n-\frac{n(n-a_{0}^2)}{2(V_{0}-a_{0}^2\ell)}x} & \textnormal{if }x\in [0,2(V_{0}-a_{0}^2\ell)/n]\\
a_{0} & \textnormal{if }x\in (2(V_{0}-a_{0}^2\ell)/n,\ell].
\end{array}\right.
$$
By construction, 
$$
\int_{0}^\ell a_{n}(x)^2\, dx=\ell a_{0}^2+\frac{(n-a_{0}^2)(V_{0}-a_{0}^2\ell)}{n}<V_{0},
$$
so that $a_{n}\in \mathcal{V}_{a_{0},\ell,V_{0}}$ for every $n\in\N^*$. Since $a_{n}\in C^0([0,\ell])$ and $a_{n}(0)=\sqrt{n}$, it follows obviously that $\mathcal{V}_{a_{0},\ell,V_{0}}$ is not bounded in $L^\infty(0,\ell)$.

Second, consider $a\in \mathcal{S}_{a_{0},\ell,S_{0}}$. One has
$$
S_{0}\geq \int_{x}^ya(t)\sqrt{1+a'(t)^2}\, dt\geq \int_{x}^ya(t)|a'(t)|\, dt\geq \left|\int_{x}^ya(t)a'(t)\, dt\right|=\frac{1}{2}|a(y)^2-a(x)^2|,
$$
for every $0<x<y<\ell$. It follows that for every $x\in [0,\ell]$,
\begin{equation}\label{ineqa}
a(0)^2-2S_{0}\leq a(x)^2\leq a(0)^2+2S_{0}.
\end{equation}
Using this inequality, one gets
$$
S_{0}\geq \int_{0}^\ell a(t)\, dt\geq \ell\sqrt{a(0)^2-2S_{0}}
$$
and thus, $a(0)^2\leq S_{0}^2/\ell^2+2S_{0}$. The conclusion follows by combining this inequality with \eqref{ineqa}.
\end{proof}
\subsection{Main results: maximizing $F$ with a volume or a surface constraint}\label{sec:mainRes}
In this section, we solve the problems \eqref{defV} and \eqref{defS} which, according to the previous sections, model the optimal shape of a axisymmetric fin with either a volume or a surface constraint. 

The following theorem highlights the ill-posed character of Problem \eqref{defV}, where a volume constraint is imposed.
\begin{theorem}\label{thpb1D:vol}
Let $a_{0}$ and $\ell$ denote two positive real numbers, and let $V_{0}>\ell a_{0}^2$. Problem \eqref{defV} has no solution and
$$
\sup_{a\in \mathcal{V}_{a_{0},\ell,V_{0}}}F(a)=+\infty.
$$
\end{theorem}

In the next theorem, we prove that imposing a lateral surface constraint on $\Omega_{a}$ is not enough to get the existence of a solution for this problem. Nevertheless, on the contrary to the previous case where a volume constraint were imposed, the value of the supremum of $F$ over the set $\mathcal{S}_{a_{0},\ell,S_{0}}$ is now finite.

\begin{theorem}\label{thpb1D:surf}
Let $a_{0}$ and $\ell$ denote two positive real numbers, and let $S_{0}> \ell a_{0}$. Problem \eqref{defS} has no solution and
$$
\sup_{a\in \mathcal{S}_{a_{0},\ell,S_{0}}}F(a)=k\pi\beta (T_{d}-T_{\infty})\left(\frac{a_{0}^{3/2}\gamma}{\sqrt{\beta}}+(S-a_{0}\ell)\right),
$$
where $\gamma$ denotes the positive real number
\begin{equation}\label{def:gamma} 
\gamma=\frac{\sqrt{\frac{\beta}{a_{0}}}\sinh\left(\sqrt{\frac{\beta}{a_{0}}}\ell\right)+\beta_{r}\cosh\left(\sqrt{\frac{\beta}{a_{0}}}\ell\right)}{\sqrt{\frac{\beta}{a_{0}}}\cosh\left(\sqrt{\frac{\beta}{a_{0}}}\ell\right)+\beta_{r}\sinh\left(\sqrt{\frac{\beta}{a_{0}}}\ell\right)}.
\end{equation}
Moreover, every sequence $(a_{n})_{n\in \N}$ of elements of $\mathcal{S}_{a_{0},\ell,S_{0}}$ such that $(a_{n},a_{n}\sqrt{1+a_{n}'^2})_{n\in \N}$ $\tau$-converges to $(a_{0},a_{0}+(S-a_{0}\ell)\delta_{0})$, where $\delta_{0}$ denotes the Dirac measure at $x=0$, is a maximizing sequence for Problem \eqref{defS}.
\end{theorem}

\begin{remark}\label{rk:rearrangement}
In the proof of Theorem \ref{thpb1D:surf}, the key idea to show the nonexistence result and to construct maximizing sequences lies in applying a particular perturbation of the function $a\sqrt{1+a'^2}$, constructed in the vein of monotone or Schwarz rearrangements but it is not one of then. We then exhibit an element $a\in \mathcal{A}_{a_{0},\ell}$ belonging to the class of admissible designs, that makes the job.

Notice that it is standard in shape optimization to use rearrangements such as the so-called {\it Steiner} or {\it Schwarz symmetrizations}, to characterize the solutions of an optimal design problem (see for instance \cite{kawohl} for the definition and \cite{henrot} for examples of use in shape optimization). Unfortunately, we did not manage to conclude the proof by applying standard rearrangement arguments, which explains why a more subtle analysis is performed to get the nonexistence result.
\end{remark}
\subsection{ Proof of Theorem \ref{thpb1D:vol} (maximization of $F$ with a volume constraint)}\label{sec:optvol}
To prove this theorem, we will exhibit an explicit maximizing sequence $(a_{n})_{n\in\N^*}$. The construction of the sequence $(a_{n})_{n\in\N^*}$ is based on the use of an intermediate sequence denoted $(a_{S,m})_{m\in\N^*}$ where $S$ denotes a positive real number. Fix $S>0$. The sequence $(a_{S,m})_{m\in\N^*}$ is chosen to satisfy at the same time
\begin{itemize}
\item $(a_{S,m})_{m\in\N^*}$ converges strongly to $a_{0}$ in $L^\infty(0,\ell)$,
\item the sequence $(b_{S,m})_{m\in\N^*}$ defined by $b_{S,m}=a_{S,m}\sqrt{1+a_{S,m}'^2}$ for every $m\in\N^*$ satisfies 
$$
\operatorname{surf}(a_{S,m})=\int_{0}^\ell b_{S,m}(x)\, dx=S
$$
for every $m\in\N^*$ and converges in the sense of measures to $a_{0}+(S-a_{0}\ell)\delta_{0}$, where $\delta_{0}$ denotes the Dirac measure at $x=0$.
\end{itemize}
In particular, the sequence $(a_{S,m},a_{S,m}\sqrt{1+a_{S,m}'^2})_{m\in\N^*}$ $\tau$-converges to $(a_{0},a_{0}+(S-a_{0}\ell)\delta_{0})$ as $m$ tends to $+\infty$.

Let us now provide an example of such sequence. A possible choice of function $b_{S,m}$ for $m$ large enough is given by
$$
b_{S,m}(x)=\left\{\begin{array}{ll}
a_{0}+(S-a_{0}\ell)m & \textnormal{if }x\in [0,1/m]\\
a_{0}& \textnormal{if }x\in (1/m,\ell].
\end{array}\right.
$$
Solving the equation $b_{S,m}=a_{S,m}\sqrt{1+a_{S,m}'^2}$ comes to solve
$
\frac{a_{S,m}|a_{S,m}'|}{\sqrt{b_{S,m}^2-a_{S,m}^2}}=1
$
on each interval $(0,1/m)$ and $(1/m,\ell)$. This equation has obviously an infinite number of solutions. In particular, a family of solutions is obtained by making the function $a_{S,m}$ oscillate an integer number of times in $(0,1/m)$, each oscillation corresponding to two successive choices of the sign of $a_{S,m}'$ on two same length intervals. The continuity of each function $a_{S,m}$ determines then it in a unique way. As a consequence, it is possible to control the $L^\infty$-norm of $a_{S,m}$ by choosing artfully the number of oscillations, as made precise below.

Introduce $M_{m}=a_{0}+(S-a_{0}\ell)m$ for every $m\in \N^*$. We now construct the sequence $(a_{S,m})_{n\in\N^*}$ so that each function oscillates $m$ times on $[0,1/m]$ (see Figure \ref{Fig:suitemax1}). More precisely, we define $a_{S,m}$ by
\begin{equation}\label{def:aSm}
a_{S,m}(x)=\left\{\begin{array}{ll}
\sqrt{M_{m}^2-(\sqrt{M_{m}^2-a_{0}^2}-x)^2} & \textrm{ on }\left[0 , \frac{1}{2m^2}\right) \ ; \\
a_{S,m}(\frac{1}{m^2}-x) & \textrm{ on }\left[ \frac{1}{2m^2},\frac{1}{m^2}\right) \ ; \\
a_{S,m}\left(x-\frac{i}{m^2}\right) & \textrm{ on }\left[\frac{i}{m^2},\frac{i+1}{m^2}\right), \  i \in \{ 1,..., m-1\} \ ; \\
a_{0} & \textrm{ on } \left[ \frac{1}{m}, \ell\right].
\end{array}
\right.
\end{equation}
Notice in particular that 
$$
\Vert a_{S,m}-a_{0}\Vert_{\infty}= a_{S,m}\left(\frac{1}{2m^2}\right)-a_{0}=\operatorname{o}\left(\frac{1}{m^2}\right)\quad \textrm{as }m\to +\infty.
$$
\begin{figure}[h!]
\begin{center}
\includegraphics[height=5.7cm]{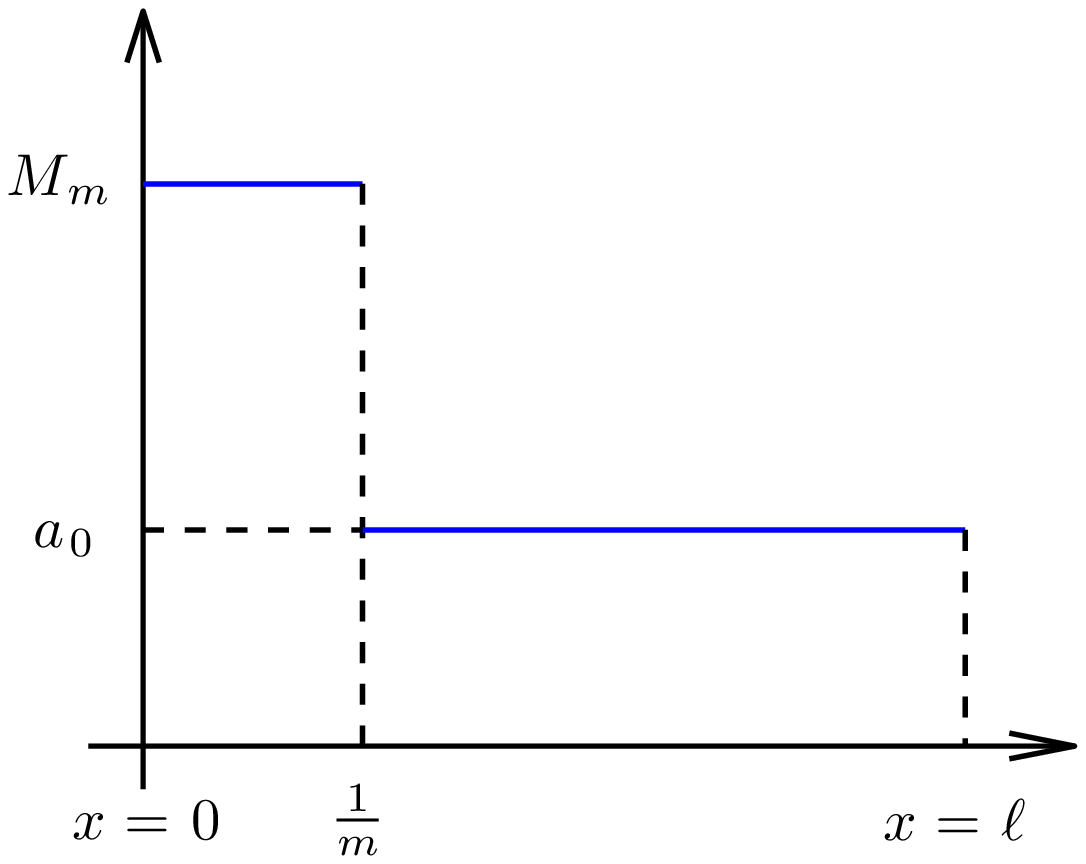}
\includegraphics[height=5cm]{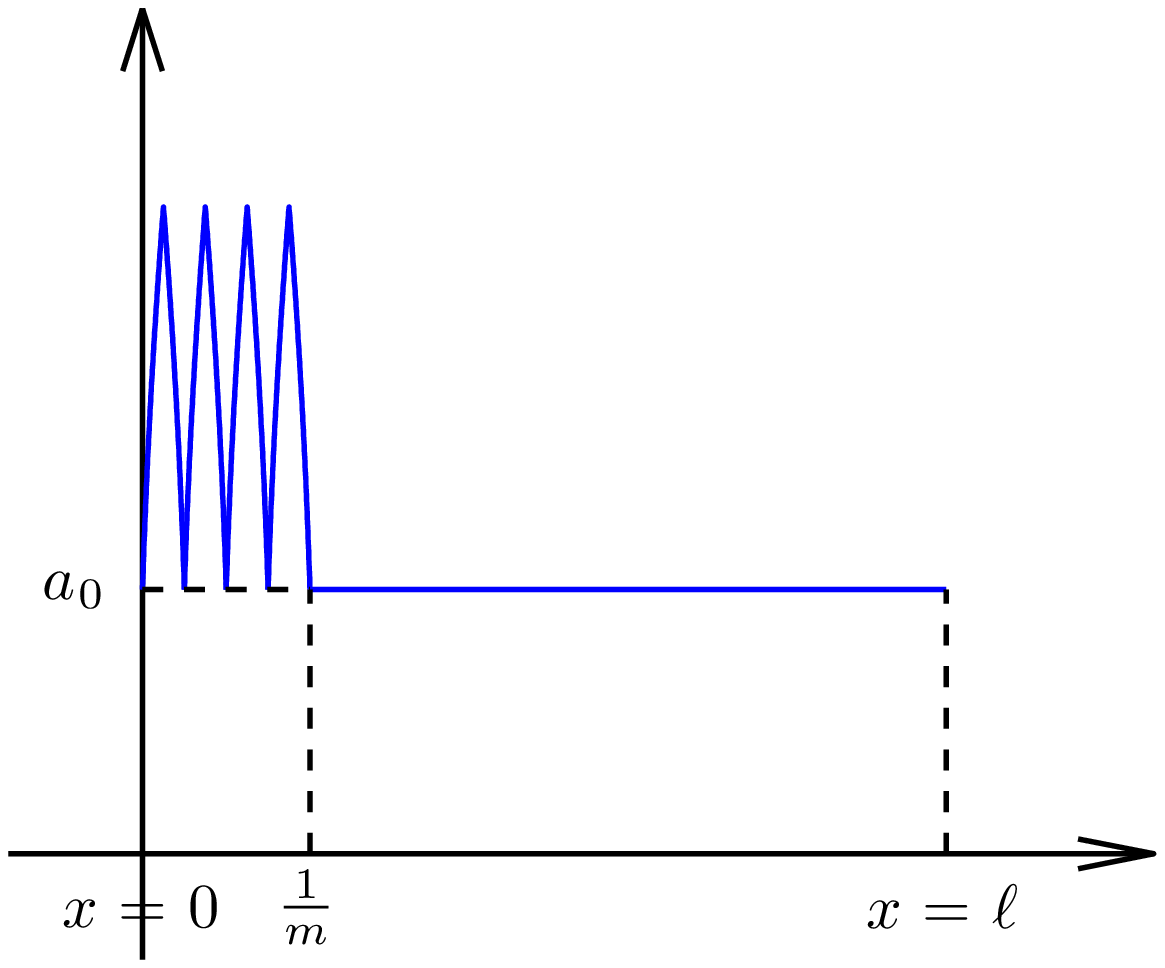}
\caption{Left: graph of $b_{S,m}=a_{S,m}\sqrt{1+a_{S,m}'^2}$. Right: Graph of $a_{S,m}$.}\label{Fig:suitemax1}
\end{center}
\end{figure}
The sequence $(a_{S,m})_{n\in\N^*}$ satisfies then the expected properties above. Moreover, according to Proposition \ref{prop:FC0}, one has
$$
\lim_{m\to +\infty} F(a_{S,m})=\lim_{m\to +\infty}\widehat{F}(a_{S,m},a_{S,m}\sqrt{1+a_{S,m}'^2})=\widehat{F}(a_{0},a_{0}+(S-a_{0}\ell)\delta_{0})
$$
Notice that the solution $T$ of \eqref{eq:Tbis} associated to the pair $(a,b)=(a_{0},a_{0}+(S-a_{0}\ell)\delta_{0})$ coincides with the solution of the same equation associated to the pair $(a,b)=(a_{0},a_{0})$. Explicit computations thus lead to
$$
T(x)=T_{\infty}+(T_{d}-T_{\infty})\left(\cosh\left(\sqrt{\frac{\beta}{a_{0}}}x\right)-\gamma \sinh\left(\sqrt{\frac{\beta}{a_{0}}}x\right)\right),
$$
for every $x\in [0,\ell]$, where $\gamma$ is defined by \eqref{def:gamma}.
We then compute
\begin{eqnarray*}
\lim_{m\to +\infty}\frac{F(a_{S,m})}{k\pi} & = & \beta\langle a_{0}+(S-a_{0}\ell) \delta_{0},T-T_{\infty}\rangle_{\mathcal{M}(0,\ell),C^0([0,\ell])}+\beta_{r}a_{0}^2(T(\ell)-T_{\infty})\\
& = & \beta (T_{d}-T_{\infty})\left(\frac{a_{0}^{3/2}}{\sqrt{\beta}}\gamma+(S-a_{0}\ell)\right).
\end{eqnarray*}
\begin{paragraph}{Conclusion: construction of a maximizing sequence $(a_{n})_{n\in\N^*}$}
Let $n\in \N^*$ such that $n\geq [a_{0}\ell]+1$. Consider the sequence $(a_{n,m})_{m\in \N^*}$ introduced previously. Since $(a_{n,m})_{m\in \N^*}$ converges strongly to $a_{0}$ in $L^\infty(0,\ell)$ and according to the previous convergence study, there exists $p\in \N$ such that $\operatorname{vol}(a_{n,p})\leq V_{0}-\frac{1}{n}$ and
$$ 
\frac{F(a_{n,p})}{k\pi \beta(T_{d}-T_{\infty})}\geq \left(\frac{a_{0}^{3/2}\gamma}{\sqrt{\beta}}+(n-a_{0}\ell)\right)-\frac{1}{n}.
$$
Let us denote by $m_{n}$ the first integer for which this property is verified, and by $a_{n}$ the function equal to $a_{0}$ if $n\leq [a_{0}\ell]$ and by $a_{n,m_{n}}$ else. It follows not only that each element $a_{n}$ belongs to $\mathcal{V}_{a_{0},\ell,V_{0}}$ and that
$$
\sup_{a\in \mathcal{V}_{a_{0},\ell,V_{0}}}\frac{F(a)}{k\pi \beta}\geq \frac{F(a_{n,m_{n}})}{k\pi \beta}\geq (T_{d}-T_{\infty})\left(\frac{a_{0}^{3/2}\gamma}{\sqrt{\beta}}+(n-a_{0}\ell)-\frac{1}{n}\right)
$$
for every $n\geq [a_{0}\ell]+1$. Letting $n$ tend to $+\infty$ yields the conclusion of the theorem.
\end{paragraph}

\subsection{Proof of Theorem \ref{thpb1D:vol} (maximization of $F$ with a lateral surface constraint)}\label{sec:optS}
This proof is divided into several steps. Steps 1, 2 and 3 are devoted to showing that Problem \eqref{defS} has no solution, whereas Step 4 focuses on building maximizing sequences by introducing a king of truncated optimal design problem.
Let us argue by contradiction, assuming that Problem \eqref{defS} has a solution $a\in \mathcal{S}_{a_{0},\ell,S_{0}}$. Let us denote by $b$ the function $a\sqrt{1+a'^2}$. The idea of the proof is to introduce an admissible perturbation $a_{\varepsilon}$ of $a$ in $\mathcal{S}_{a_{0},\ell,S_{0}}$, suitably chosen to guarantee that $F(a_{\varepsilon})>F(a)$ for a given $\varepsilon>0$.  We will proceed in several steps.
\begin{paragraph}{Step 1. Definition of the perturbation $a_{\varepsilon}$}
Since the constant function equal to $a_{0}$ is obviously not a solution of Problem \eqref{defS}, there exists $x_{0}\in (0,\ell)$ and an interval $(x_{0}-\varepsilon/2,x_{0}+\varepsilon/2)$ on which $a>a_{0}$.
Introduce for such a choice of $x_{0}$ and $\varepsilon$, the function
$$
b_{\varepsilon}=b+c(\chi_{[0,\varepsilon]}-\chi_{[x_{0}-\varepsilon/2,x_{0}+\varepsilon+2]}),
$$
where $c$ is a positive constant, chosen small enough to guarantee that $b_{\varepsilon}>a_{0}$ almost everywhere in $(0,\ell)$. We will use the following lemma to construct $a_{\varepsilon}$.
\begin{lemma}\label{lemma:thm2}
There exists a family $(a_{\varepsilon})_{\varepsilon>0}$ such that $a_{\varepsilon}\in \mathcal{S}_{a_{0},\ell,S_{0}}$ for every $\varepsilon>0$,  $a_{\varepsilon}\sqrt{1+a_{\varepsilon}'^2}=b_{\varepsilon}$ for almost every $\varepsilon>0$, and
$$
\Vert a_{\varepsilon}-a\Vert_{L^\infty(0,\ell)}=\operatorname{O}(\varepsilon^2).
$$
\end{lemma}
\begin{proof}
The proof consists in exhibiting an element $a_{\varepsilon}\in \mathcal{S}_{a_{0},\ell,S_{0}}$ so that the statements of Lemma \ref{lemma:thm2} are satisfied. First, we impose $a_{\varepsilon}=a$ on $(\varepsilon,x_{0}-\varepsilon/2)\cup(x_{0}+\varepsilon/2,\ell)$. Without loss of generality, let us now explain how to define $a_{\varepsilon}$ on $[0,\varepsilon]$, the construction of $a_{\varepsilon}$ on $[x_{0}-\varepsilon/2,x_{0}+\varepsilon/2]$ being similar. The construction method is close to the one of the maximizing sequence presented in the proof of Theorem \ref{thpb1D:vol}. The idea is to impose oscillations on $a_{\varepsilon}$ to control the $L^\infty$ distance between $a$ and $a_{\varepsilon}$. 

We first explain how to create one oscillation on an interval $[\bar x,\bar x+\eta]$, where $0<\bar x<\bar x+\eta < \varepsilon$ (see Figure \ref{fig:oscil}). The function $a_{\varepsilon}$ is chosen so that
$$
a_{\varepsilon}=a_{\eta,1}\quad \textnormal{on }(\bar x,\xi)\qquad\textnormal{and}\qquad a_{\varepsilon}=a_{\eta,2}\quad\textnormal{on }(\xi,\bar x+\eta),
$$
where the function $a_{\eta,1}$ is solution of the following Cauchy problem
$$
\begin{array}{ll}
a_{\eta,1}'(x)=\frac{\sqrt{b_{\varepsilon}(x)^2-a_{\eta,1}(x)^2}}{a_{\eta,1}(x)} & x\in (\bar x, \bar x+\eta)\\
a_{\eta,1}(\bar x)=a(\bar x), & 
\end{array}
$$
the function $a_{\eta,2}$ is solution of the following Cauchy problem
$$
\begin{array}{ll}
a_{\eta,2}'(x)=-\frac{\sqrt{b_{\varepsilon}(x)^2-a_{\eta,2}(x)^2}}{a_{\eta,2}(x)} & x\in (\bar x, \bar x+\eta)\\
a_{\eta,2}(\bar x+\eta)=a(\bar x+\eta), & 
\end{array}
$$
and $\xi \in (\bar x,\bar x+\eta)$ is chosen so that $a_{\eta,1}(\xi)=a_{\eta,2}(\xi)$. Notice that the functions $a_{\eta,1}$ and $a_{\eta,2}$ 
satisfy in particular $a_{\eta,i}\sqrt{1+a_{\eta,i}'^2}=b_{\varepsilon}$ for $i\in\{1,2\}$, $a_{\eta,1}$ is increasing and $a_{\eta,2}$ is decreasing. 
Such a construction is possible provided that the graphs of $a_{\eta,1}$ and $a_{\eta,2}$ intersect at a point whose abscissa belongs to $(\bar x,\bar x+\eta)$. 
The intermediate value theorem yields that it is enough to show that
$$
a_{\eta,1}(\bar x+\eta)>a(\bar x+\eta)\qquad \textnormal{and}\qquad a_{\eta,2}(\bar x)>a(\bar x).
$$
\begin{figure}[h!]
\begin{center}
\includegraphics[width=7cm]{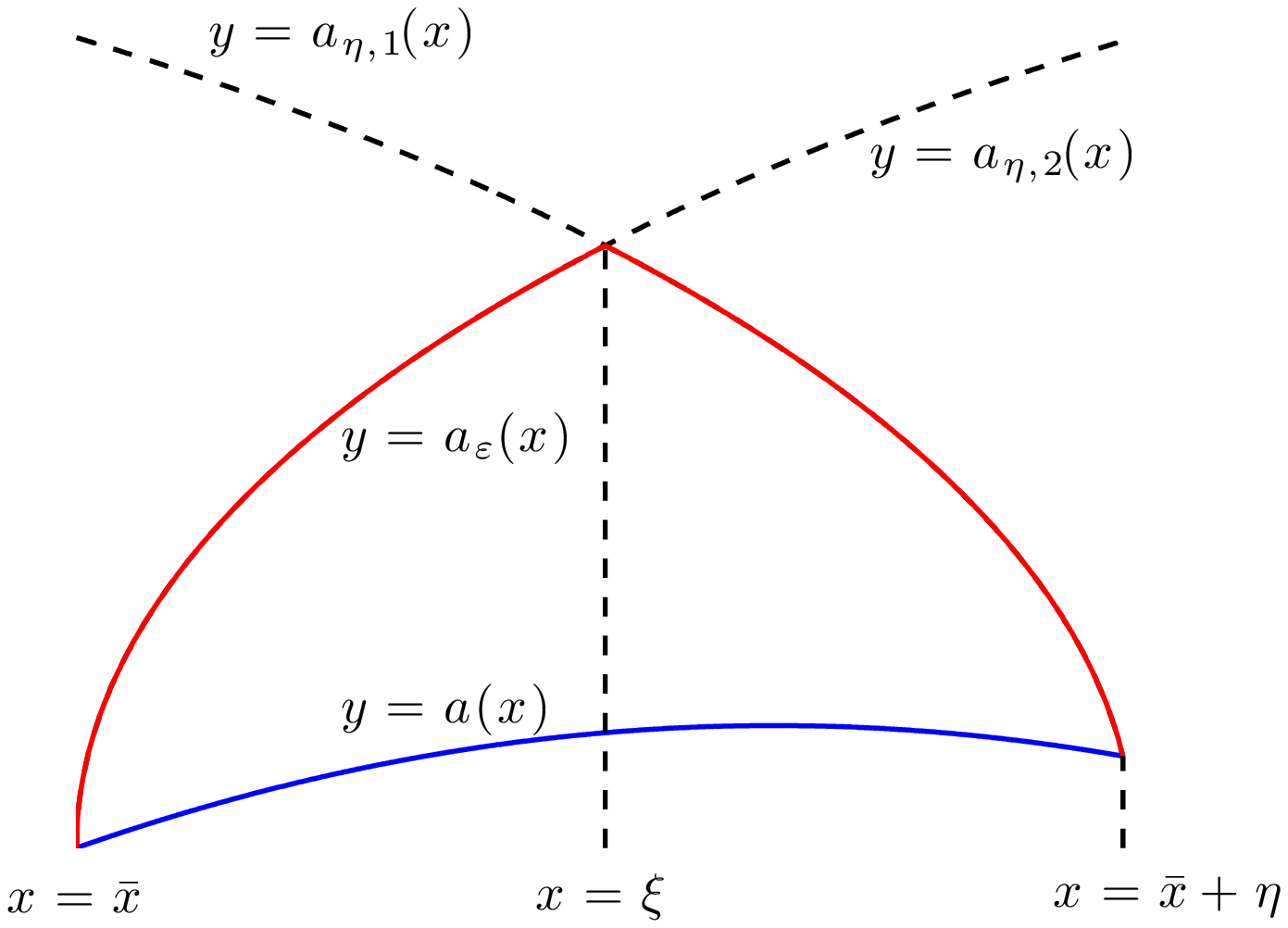}
\includegraphics[width=7cm]{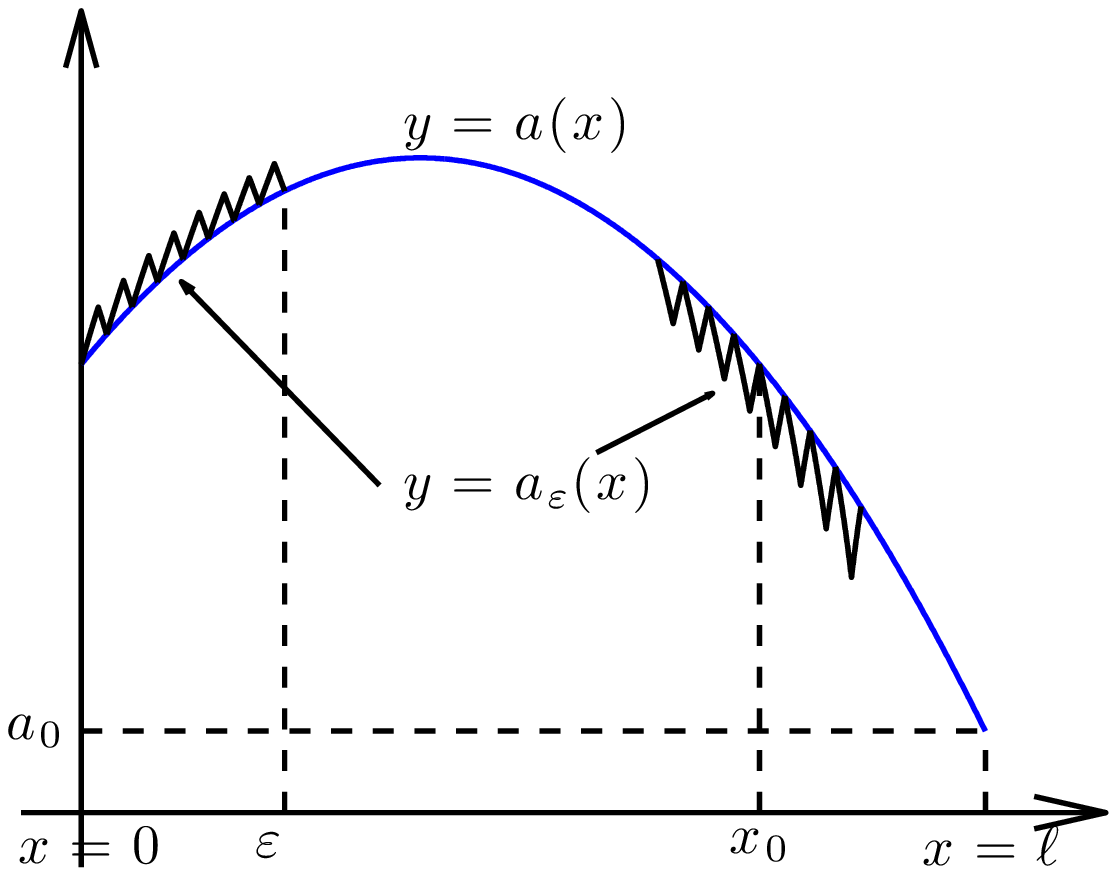}
\caption{Left: Zoom on one oscillation. Right: the perturbation $a_{\varepsilon}$}\label{fig:oscil}
\end{center}
\end{figure}

Without loss of generality, let us prove that the first assertion is true, the proof of the second one being similar. 
In fact, we will prove that $a_{\eta,1}>a$ everywhere in $(\bar x,\bar x+\eta]$. Let us argue by contradiction. 
Note that $a_{\eta,1}'(\bar x)>a'(\bar x)$ since $\bar{x}<\varepsilon$ and thus $b_\varepsilon (\bar{x})>b(\bar{x})$. 
Denote by $\alpha$ the first point of $(\bar x,\bar x+\eta)$ such that $a_{\eta,1}(\alpha)=a(\alpha)$, provided that it exists. 
Since $b_\varepsilon>b$ on $[\bar{x},\bar{x}+\eta]\subset [0,\varepsilon]$, there holds $a_{\eta,1}\sqrt{1+a_{\eta,1}'^2}>a\sqrt{1+a'^2}$ almost everywhere in any neighborhood 
of $\alpha$, thus there exists necessarily a neighborhood $\mathcal{O}_{\alpha}$ of $\alpha$ such that $|a_{\eta,1}'(x)|>|a'(x)|$ at every Lebesgue point of $a_{\eta,1}'$ and $a'$ in $\mathcal{O}_{\alpha}$. It implies the existence of $\nu>0$ such that $a_{\eta,1}(x)<a(x)$ on $(\alpha-\nu,\alpha)$, which is absurd. 

It remains now to find the number of oscillation on $[0,\varepsilon]$ and $[x_{0}-\varepsilon/2,x_{0}+\varepsilon/2]$ guaranteeing that $a_{\varepsilon}$ is as close to $a$ as desired (in the sense of the $L^\infty$ distance). Notice that, according to the previous definition of $a_{\varepsilon}$ on $[\bar x,\bar x+\eta]$, one has
\begin{eqnarray*}
a(x)^2\leq a_{\varepsilon}(x)^2& \leq & 2\int_{\bar x}^{\bar x+\eta} b_{\varepsilon}(x) + a(\bar{x})^2 \\ & \leq &  2\eta(\Vert b+c\Vert_{L^\infty(0,\ell)} +2\Vert a'\Vert_{L^\infty(0,\ell)}\Vert a\Vert_{L^\infty(0,\ell)})+a(x)^2,
\end{eqnarray*}
for every $x\in [\bar x,\bar x+\eta]$. Moreover, since $a\in W^{1,\infty}(0,\ell)$, 
$$
0\leq a_{\varepsilon}(x)-a(x)\leq \frac{(\Vert b+c\Vert_{L^\infty(0,\ell)}+\Vert a'\Vert_{L^\infty(0,\ell)}\Vert a\Vert_{L^\infty(0,\ell)})}{a_{0}}\eta
$$
for every $x\in [\bar x,\bar x+\eta]$. If suffices hence to fix for example $\eta=\varepsilon/k_{\varepsilon}$ with $k_{\varepsilon}=\left[1/\varepsilon\right]+1$, to ensure that $\Vert a_{\varepsilon}-a\Vert _{L^\infty ([\bar x,\bar x+\eta])}=\operatorname{O}(\varepsilon^2)$.

\begin{paragraph}{Conclusion} 
We consider a regular subdivision of $[0,\varepsilon]$ (resp. $[x_{0}-\varepsilon/2,x_{0}+\varepsilon/2]$) into $k_{\varepsilon}$ intervals, 
and we use the process described previously to construct the graph of $a_{\varepsilon}$ by creating one oscillation on each of these intervals. 
This way, one gets: $\Vert a_{\varepsilon}-a\Vert _{L^\infty ([0,\ell ])}=\operatorname{O}(\varepsilon^2)$.
\end{paragraph}
The lemma is then proved.
\end{proof}

We now consider a family $(a_{\varepsilon})_{\varepsilon>0}$ chosen as in the statement of Lemma \ref{lemma:thm2}. In particular, and according to the construction of $a_{\varepsilon}$ made in the proof of Lemma \ref{lemma:thm2}, we will assume without loss of generality that $a_{\varepsilon}(\ell)=a(\ell)$. 

For $\varepsilon>0$, we denote by $T_{\varepsilon}$ the temperature associated to $a_{\varepsilon}$, \textit{i.e.} the solution of \eqref{eq:T} with $a=a_{\varepsilon}$.
\end{paragraph}
\begin{paragraph}{Step 2. Asymptotic development of $T_{\varepsilon}$ at the first order}
Noticing that $\int_{0}^\ell b_{\varepsilon}(x)\, dx=\int_{0}^\ell b(x)\, dx$, $(b_{\varepsilon})_{\varepsilon>0}$ converges in the sense of measures to $b$, and $(a_{\varepsilon})_{\varepsilon>0}$ converges strongly to $a$ in $L^\infty(0,\ell)$  as $\varepsilon\searrow 0$, according to Lemma \ref{lemma:thm2}. Then, following the proof of Proposition \ref{prop:FC0}, the family $(T_{\varepsilon})_{\varepsilon>0}$ converges, up to a subsequence, weakly in $H^1(0,\ell)$ and strongly in $L^2(0,\ell)$ to $T$, the solution of \eqref{eq:T} associated to the optimal radius $a$.

Let us write an asymptotic development of $T_{\varepsilon}$ at the first order. For that purpose, introduce
$$
\widetilde{T}_{\varepsilon}=\frac{T_{\varepsilon}-T}{\varepsilon}.
$$
The function $\widetilde{T}_{\varepsilon}$ is solution of the following ordinary differential equation
\begin{equation}\label{eq:Teps}
\begin{array}{ll}
(a(x)^2\widetilde{T}_{\varepsilon}'(x))'= \beta b(x)\widetilde{T}_{\varepsilon}
+\beta \frac{R_{\varepsilon}(x)}{\varepsilon}(T_{\varepsilon}(x)-T_{\infty})-\left(\left(\frac{a_{\varepsilon}(x)^2-a(x)^2}{\varepsilon}\right)T_{\varepsilon}'(x)\right)',
& x\in (0,\ell)\\
\widetilde{T}_{\varepsilon}(0)=0 & \\
\widetilde{T}_{\varepsilon}'(\ell)=-\beta_{r}\widetilde{T}_{\varepsilon}(\ell), & 
\end{array}
\end{equation}
where $R_{\varepsilon}(x)=c(\chi_{[0,\varepsilon]}-\chi_{[x_{0}-\varepsilon/2,x_{0}+\varepsilon+2]})$. 

Let us first prove the convergence of $\widetilde{T}_{\varepsilon}$. We multiply the main equation of System \eqref{eq:Teps} by $\widetilde{T}_{\varepsilon}$ and then integrate. We get
\begin{eqnarray*}
\int_{0}^\ell \left(a(x)^2\widetilde{T}_{\varepsilon}'(x)^2+\beta b(x)\widetilde{T}_{\varepsilon}(x)^2\right)\, dx + \beta_{r}a(\ell)^2\widetilde{T}_{\varepsilon}(\ell)^2 &=& -\frac{\beta}{\varepsilon}\int_{0}^\varepsilon R_{\varepsilon}(x)\widetilde{T}_{\varepsilon}(x)(T_{\varepsilon}(x)-T_{\infty})\, dx\\
& & -\int_{0}^\ell \left(\frac{a_{\varepsilon}(x)^2-a(x)^2}{\varepsilon}\right)T_{\varepsilon}'(x) \widetilde{T}_{\varepsilon}'(x)\, dx.
\end{eqnarray*}
First, notice that, using standard Sobolev imbedding results and the weak-$H^1$ convergence of the family $(T_{\varepsilon})_{\varepsilon>0}$, there exists $C_{1}>0$ such that 
\begin{eqnarray*}
\frac{\beta}{\varepsilon}\left|\int_{0}^\varepsilon R_{\varepsilon}(x)\widetilde{T}_{\varepsilon}(x)(T_{\varepsilon}(x)-T_{\infty})\, dx\right| & = &
\frac{\beta c}{\varepsilon}\left|\int_{0}^\varepsilon \widetilde{T}_{\varepsilon}(T_{\varepsilon}-T_{\infty})\, dx-\int_{x_{0}-\varepsilon/2}^{x_{0}+\varepsilon/2} \widetilde{T}_{\varepsilon}(T_{\varepsilon}-T_{\infty})\, dx\right|\\
& \leq & 2\beta c\Vert \widetilde{T}_{\varepsilon}\Vert_{L^\infty(0,\ell)}(\Vert T_{\varepsilon}\Vert_{L^\infty(0,\ell)}+T_{\infty})\\
&\leq & C_{1}\Vert \widetilde{T}_{\varepsilon}\Vert_{H^1(0,\ell)}.
\end{eqnarray*}
Second, using at the same time Lemma \ref{lemma:10H28}, \ref{lemma:thm2} and the Cauchy-Schwarz inequality, there exists $C_{2}>0$ such that
\begin{eqnarray*}
\left|\int_{0}^\ell \left(\frac{a_{\varepsilon}(x)^2-a(x)^2}{\varepsilon}\right)T_{\varepsilon}'(x) \widetilde{T}_{\varepsilon}'(x)\, dx \right| & \leq &
2\sqrt{\frac{S_{0}^2}{\ell^2}+4S_{0}}\frac{\Vert a_{\varepsilon}-a\Vert_{L^\infty(0,\ell)}}{\varepsilon} \Vert \widetilde{T}_{\varepsilon}\Vert_{H^1(0,\ell)}\Vert T_{\varepsilon}-T_{\infty}\Vert_{H^1(0,\ell)} \\
&\leq & C_{2}\varepsilon \Vert \widetilde{T}_{\varepsilon}\Vert_{H^1(0,\ell)}.
\end{eqnarray*}
Combining the two previous estimates and using that
$$
a_{0}\min \{\beta,a_{0}\}\Vert \widetilde{T}_{\varepsilon}\Vert_{H^1(0,\ell)}^2 \leq \int_{0}^\ell \left(a(x)^2\widetilde{T}_{\varepsilon}'(x)^2+\beta b(x)\widetilde{T}_{\varepsilon}(x)^2\right)\, dx+ \beta_{r}a(\ell)^2\widetilde{T}_{\varepsilon}(\ell)^2,
$$
yields that $(\widetilde{T}_{\varepsilon})_{\varepsilon>0}$ is bounded in $H^1(0,\ell)$. Then, using a Rellich theorem, $(\widetilde{T}_{\varepsilon})_{\varepsilon>0}$ converges, up to a subsequence, to $\widetilde{T}$, weakly in $H^1(0,\ell)$ and strongly in $L^2(0,\ell)$. 

Let us now write the system whose $\widetilde{T}$ is solution. The variational formulation of \eqref{eq:Teps} writes: find $\widetilde{T}_{\varepsilon}\in H^1(0,\ell)$ satisfying $\widetilde{T}_{\varepsilon}(0)=0$ such that for every test function $\varphi\in H^1(0,\ell)$ satisfying $\varphi(0)=0$, one has
\begin{eqnarray}
\int_{0}^\ell \left(a(x)^2\widetilde{T}_{\varepsilon}'(x)\varphi'(x)+\beta b(x)\widetilde{T}_{\varepsilon}(x)\varphi(x)\right)\, dx +\beta_{r}a(\ell)^2\widetilde{T}_{\varepsilon}(\ell)\varphi(\ell)\nonumber \\
+ \beta \int_{0}^\ell  \frac{R_{\varepsilon}(x)}{\varepsilon}(T_{\varepsilon}(x)-T_{\infty})\varphi(x)\, dx -\int_{0}^\ell \left(\left(\frac{a_{\varepsilon}(x)^2-a(x)^2}{\varepsilon}\right)T_{\varepsilon}'(x)\right)' \varphi(x)\, dx=0.\label{FVTtildeps}
\end{eqnarray}
Note that
$$
-\int_{0}^\ell \left(\left(\frac{a_{\varepsilon}(x)^2-a(x)^2}{\varepsilon}\right)T_{\varepsilon}'(x)\right)' \varphi(x)\, dx= \int_{0}^\ell \left(\frac{a_{\varepsilon}(x)^2-a(x)^2}{\varepsilon}\right)T_{\varepsilon}'(x)\varphi'(x)\, dx.
$$
Then, since $(T_{\varepsilon})_{\varepsilon>0}$ is uniformly bounded in $H^1(0,\ell)$ and using Lemma \ref{lemma:thm2}, one deduces that
$$
\lim_{\varepsilon\searrow 0}\int_{0}^\ell \left(\left(\frac{a_{\varepsilon}(x)^2-a(x)^2}{\varepsilon}\right)T_{\varepsilon}'(x)\right)' \varphi(x)\, dx=0.
$$
We let $\varepsilon$ tend to zero in \eqref{FVTtildeps}, and using the Lebesgue density theorem, we get
\begin{equation}\label{FVTtilde}
\int_{0}^\ell \left(a(x)^2\widetilde{T}'(x)\varphi'(x)+\beta b(x)\widetilde{T}(x)\varphi(x)\right)\, dx-\beta c\varphi(x_{0})(T(x_{0})-T_{\infty})+\beta_{r}a(\ell)^2\widetilde{T}(\ell)\varphi(\ell)=0,
\end{equation}
for every test function $\varphi\in H^1(0,\ell)$ satisfying $\varphi(0)=0$. We refer to Remark \ref{RkFVTtilde} for the characterization of $\widetilde{T}$.
\end{paragraph}
\begin{paragraph}{Step 3. Asymptotic of the cost functional \eqref{def:F(a)} when $\varepsilon$ tends to 0 and conclusion}
We compute
\begin{eqnarray*}
\frac{F(a_{\varepsilon})-F(a)}{k\pi \beta} & = & \int_{0}^\ell \left(b_{\varepsilon}(x)(T_{\varepsilon}(x)-T_{\infty})-b(x)(T(x)-T_{\infty})\right)\, dx\\
& & +\frac{\beta_{r}}{\beta}a(\ell)^2(T_{\varepsilon}(\ell)-T(\ell))\\
& = & \int_{0}^\ell b(x)(T_{\varepsilon}(x)-T(x))\, dx + \int_{0}^\ell R_{\varepsilon}(x)(T_{\varepsilon}(x)-T_{\infty})\, dx\\
& & +\frac{\beta_{r}}{\beta}a(\ell)^2(T_{\varepsilon}(\ell)-T(\ell)).
\end{eqnarray*}
Dividing by $\varepsilon$ and letting $\varepsilon$ go to zero, one sees that
\begin{equation}\label{diffCrit}
\lim_{\varepsilon\searrow 0}\frac{F(a_{\varepsilon})-F(a)}{k\pi \beta\varepsilon}=\int_{0}^\ell b(x)\widetilde T(x)\, dx+c(T_{d}-T(x_{0})) +\frac{\beta_{r}}{\beta}a(\ell)^2\widetilde{T}(\ell).
\end{equation}
The contradiction will follow by proving that the right hand-side of the last equality is positive. For that purpose, take $\varphi=T-T_{d}$ in \eqref{FVTtilde}. We obtain
\begin{eqnarray}\label{eqFV1}
\int_{0}^\ell \left(a(x)^2\widetilde T'(x)T'(x)+\beta b(x)\widetilde T(x)T(x)\right)\, dx& = & \beta T_{d}\int_{0}^\ell b(x)\widetilde T(x)\, dx +\beta c(T(x_{0})-T_{d})(T(x_{0})-T_{\infty})\nonumber \\
 & & -\beta_{r}a(\ell)^2\widetilde{T}(\ell)(T(\ell)-T_{d}).
\end{eqnarray}
Multiply now Equation \eqref{eq:T} by $\widetilde T$ and integrate then by parts. We obtain
\begin{equation}\label{eqFV2}
\int_{0}^\ell \left(a(x)^2\widetilde T'(x)T'(x)+\beta b(x)\widetilde T(x)T(x)\right)\, dx=\beta T_{\infty}\int_{0}^\ell b(x)\widetilde T(x)\, dx-\beta_{r}a(\ell)^2\widetilde{T}(\ell)(T(\ell)-T_{\infty}).
\end{equation}
Combining \eqref{eqFV1} and \eqref{eqFV2} yields
$$
\int_{0}^\ell b(x)\widetilde T(x)\, dx= c\frac{(T_{d}-T(x_{0}))(T(x_{0})-T_{\infty})}{T_{d}-T_{\infty}}-\frac{\beta_{r}}{\beta}a(\ell)^2\widetilde{T}(\ell).
$$
According to \eqref{diffCrit}, we compute
\begin{eqnarray*}
\lim_{\varepsilon\searrow 0}\frac{F(a_{\varepsilon})-F(a)}{k\pi \beta\varepsilon} & = & c\frac{(T_{d}-T(x_{0}))(T(x_{0})-T_{\infty})}{T_{d}-T_{\infty}}+c(T_{d}-T(x_{0}))\\
& = & c\frac{(T_{d}-T_{\infty})^2-(T(x_{0})-T_{\infty})^2}{T_{d}-T_{\infty}}.
\end{eqnarray*}
According to Lemma \ref{lemma1:T}, the right hand sign is positive, and it follows that, for $\varepsilon$ small enough, $F(a_{\varepsilon})>F(a)$. This is a contradiction, and it proves that Problem \eqref{defS} has no solution.
\end{paragraph}
\begin{paragraph}{Step 4. Convergence along maximizing sequences} It remains now to prove the second claim of Theorem \ref{thpb1D:surf}. To this aim, we will use an auxiliary optimal design problem obtained from Problem \eqref{defS} by imposing a uniform upper bound on $a\sqrt{1+a'^2}$. Indeed, define for  $M>a_{0}$ the set
$$
 \mathcal{S}_{a_{0},\ell,S_{0}}^M=\left\{a\in \mathcal{S}_{a_{0},\ell,S_{0}}, \ a_0\leq a\sqrt{1+a'^2}\leq M\textrm{ a.e. in }(0,\ell)\right\}.
 $$
Notice that for a given $M>a_{0}$, the class $\mathcal{S}_{a_{0},\ell,S_{0}}^M$ is compact in $W^{1,\infty}(0,\ell)$. For $M>a_{0}$, let us introduce the auxiliary problem
\begin{equation}\label{pbM}
\sup \left\{F(a), \ a\in \mathcal{S}_{a_{0},\ell,S_{0}}^M\right\}.
\end{equation}
In the following proposition, we perform a precise analysis of this problem.
\begin{proposition}\label{prop:bang}
Let $a_{0}$, $\ell$ and $M>0$ denote three positive real numbers, with $S_{0}> \ell a_{0}$ and $M>a_0$.
Then, Problem \eqref{pbM} has a solution $\underline{a}_M$, satisfying necessarily (up to a zero Lebesgue measure subset)
\begin{equation} \label{eq:bang}
\underline{a}_M (x)\sqrt{1+\underline{a}_M'^2(x)} = \left\{ \begin{array}{rl} M,& x\in (0,x_M),\\
                                       a_0,& x\in (x_M,\ell),\\
                                      \end{array}\right.\end{equation}
with $x_M= \frac{S-a_{0}\ell}{M-a_{0}}$. 
\end{proposition}
\begin{proof}
Consider a maximizing sequence $(a_n)_{n\in\N}$ in $\mathcal{S}_{a_{0},\ell,S_{0}}$, with $b_n= a_n \sqrt{1+a_n'^2}$ satisfying $a_0\leq b_n\leq M$ for almost every $n\in\N$. 
Clearly $a_0 |a_n'(x)|\leq M$ for almost every $n\in\N$ and $x\in (0,\ell)$. Hence $(a_n)_{n\in\N}$ is uniformly Lipschitz-continuous and bounded. According to the Arzel\`a-Ascoli theorem,
the sequence $(a_n)_{n\in\N}$ converges, up to a subsequence, to some Lipschitz-continuous limit $\underline{a}_M$, satisfying $a_0\leq \underline{b}_M\leq M$ where $\underline{b}_{M}=\underline{a}_M\sqrt{1+\underline{a}_M'^2}$. 

On the other hand, Proposition \ref{prop:FC0} yields 
$$\lim_{n\to +\infty} F(a_n)=\widehat{F}(\underline{a}_M,\underline{b}_M)=F(\underline{a}_M),$$ since $ \underline{b}_M=\underline{a}_M\sqrt{1+\underline{a}_M'^2}$. 
Hence $\underline{a}_M$ solves Problem \eqref{pbM}. 

Assume by contradiction that $\underline{a}_M$ does not satisfy (\ref{eq:bang}). Then there exist $0<y_0<x_0<\ell$ and $\varepsilon>0$ such that the function $b_{\varepsilon}$ defined by
$$b_{\varepsilon}=\underline{b}_M+c(\chi_{[y_{0},y_{0}+\varepsilon]}-\chi_{[x_{0}-\varepsilon,x_{0}]}),$$
satisfies to $a_0\leq b_\varepsilon \leq M$ almost everywhere in $[0,\ell]$. 
The same arguments as in the proof of Theorem \ref{thpb1D:surf} yield that one can construct a family $(a_\varepsilon)_{\varepsilon >0}$ in $\mathcal{S}_{a_{0},\ell,S_{0}}$ such that 
$b_\varepsilon = a_\varepsilon \sqrt{1+ a_\varepsilon'^2}$ for all $\varepsilon>0$ and $\|\underline{a}_M - a_\varepsilon\|_{L^\infty (0,\ell)}=\operatorname{O}(\varepsilon^2)$. Moreover, one computes
$$
\lim_{\varepsilon\searrow 0}\frac{F(a_{\varepsilon})-F(\underline{a}_M)}{k\pi \beta \varepsilon} = c\frac{(T_M(y_0)-T_{\infty})^2-(T_M(x_{0})-T_{\infty})^2}{T_{d}-T_{\infty}}
$$
where $T_M$ is the solution of \eqref{eq:T} associated with $\underline{a}_M$, and is decreasing according to Lemma \ref{lemma1:T}. Thus the right hand-side is positive and 
$F(a_\varepsilon)>F(\underline{a}_M)$ provided that $\varepsilon$ be small enough. It yields a contradiction since $a_\varepsilon$ belongs to the class of admissible functions for 
the current maximization problem. It follows that necessarily, the function $\underline{a}_M$ satisfies \eqref{eq:bang}, whence the result. 
\end{proof}

Next lemma highlights how solutions of Problem \eqref{pbM} can be used to exhibit a maximizing sequence for Problem \eqref{defS}.
\begin{lemma}\label{lemma:gammaCV}
The family $(\underline{a}_M)_{M>a_0}$ maximizes $F$.
\end{lemma}
\begin{proof}
Notice that constructed as well, the sequence of sets $(\mathcal{S}_{a_{0},\ell,S_{0}}^M)_{M>a_{0}}$ is increasing for the inclusion, and there holds
$$
\mathcal{S}_{a_{0},\ell,S_{0}}=\bigcup_{M>a_{0}}  \mathcal{S}_{a_{0},\ell,S_{0}}^M.
$$
It thus follows that
$$
\sup_{a\in \mathcal{S}_{a_{0},\ell,S_{0}}}F(a)=\sup_{M>a_{0}}\sup_{a\in \mathcal{S}^M_{a_{0},\ell,S_{0}}}F(a)=\lim_{M\to +\infty}\max_{a\in \mathcal{S}_{a_{0},\ell,S_{0}}^M}F(a)=\lim_{M\to +\infty}F(\underline{a}_{M}),
$$
whence the convergence of the maxima. As a result, $(\underline{a}_{M})_{M>a_{0}}$ is a maximizing family.
\end{proof}
\end{paragraph}

According to this result, consider such a family $(\underline{a}_M)_{M>a_0}$. Notice that the families $(\underline{a}_M)_{M>a_{0}}$ and $(\underline{b}_M)_{M>a_{0}}$ converge respectively uniformly in $(0,\ell]$ to $a_0$ and 
for the measures topology to a Dirac mass at $x=0$, as $M\to +\infty$. Indeed, the convergence of $(\underline{b}_M)_{M>a_{0}}$ is clear and one has 
$\underline{b}_M (x) = a_0 = \underline{a}_{M}(x)\sqrt{1+\underline{a}_M'(x)^2} \geq a_0 $ for all $x\in [x_M,\ell]$, implying the $L^\infty$-convergence of $\underline{a}_M(\cdot)$ in $(0,\ell]$ to $a_{0}$, since $x_M\to 0$ as $M\to +\infty$. 
Proposition \ref{prop:FC0} thus yields that $(F(\underline{a}_M) )_{M>a_{0}}$ converges to $\widehat{F} \big(a_0, (S-a_0\ell)\delta_0+a_0\big)$ as $M\to +\infty$. In other words, using the same computations as those in the proof of Theorem \ref{thpb1D:vol}, one gets 
$$
\lim_{M\to +\infty} F(\underline{a}_M)= 2h\pi (T_{d}-T_{\infty})\left(\frac{a_{0}^{3/2}\gamma}{\sqrt{\beta}}+S_{0}-a_{0}\ell\right).
$$
Moreover, each sequence $(a_n)_{n\in\N}$ fulfilling the conditions of Theorem \ref{thpb1D:surf} admits the same limit by Proposition \ref{prop:FC0}. 

\begin{remark}\label{RkFVTtilde}
System \eqref{FVTtilde} must be understood in the following sense: $\widetilde T$ solves the system
$$ \left\{
\begin{array}{ll}
(a(x)^2\widetilde T'(x))'=\beta b(x)\widetilde T(x) & x\in (0,x_{0})\\
(a(x)^2\widetilde T'(x))'=\beta b(x)\widetilde T(x) & x\in (x_{0},\ell)\\
\widetilde T(0)=0 & \\
\sigma (a(x_{0})^2\widetilde{T}'(x_{0}))=\beta c(T(x_{0})-T_{\infty}) & \\
\widetilde{T}'(\ell)=-\beta_{r}\widetilde{T}(\ell), & 
\end{array}\right.
$$
where $\sigma (a(x_{0})^2\widetilde{T}'(x_{0}))=a(x_{0}^+)^2\widetilde{T}'(x_{0}^+)-a(x_{0}^-)^2\widetilde{T}'(x_{0}^-)$ denotes the discontinuity jump of the function $a^2\widetilde{T}'$ at $x=x_{0}$.
\end{remark}
\section{Comments and conclusion}\label{sec:ccl}
\subsection{Extension of Theorems \ref{thpb1D:vol} and \ref{thpb1D:surf} to a more general setting}\label{sec:gal_model}
In this section we provide some generalizations of the results stated in Theorems \ref{thpb1D:vol} and \ref{thpb1D:surf} to a more general model. Indeed, the convection coefficient $h$ strongly depends on the operating conditions and geometries driving the fluid flow around the fin. Consequently, $h$ is now assumed to be a function of $x$, especially because the lower part of the fin lies in the boundary layer of the fluid characterized by lower velocities.
Therefore, the temperature $T$ along the fin is now assumed to solve the following ordinary differential equation,
\begin{equation}\label{eq:Tbis2}
\begin{array}{ll}
(a^2(x)T'(x))'= \beta(x) a(x)\sqrt{1+a'(x)^2}(T(x)-T_{\infty}) & x\in (0,\ell)\\
T(0)=T_{d} & \\
T'(\ell) = -\beta_r(T(\ell)-T_{\infty}), & 
\end{array}
\end{equation}
where $T_{d}$, $T_{\infty}$ and $\beta_{r}$ are chosen as in Section \ref{sec:model1D}, and $\beta$ denotes a nonnegative continuous function satisfying
\begin{equation}\label{assump:beta0}
\beta(x)\geq \beta_{0}>0
\end{equation}
for every $x\in [0,\ell]$ so that the fin surface cannot be insulated. Notice that, in this case, the statement of Lemma \ref{lemma1:T} still holds for the solution $T(\cdot)$ of \eqref{eq:Tbis2}. 

We investigate here the optimal design problems, generalizing those introduced in Section \ref{sec:odp1}.

\begin{quote}
\noindent{\bf Generalized optimal design problem with volume or lateral surface constraint.}
\textit{Let $a_{0}$ and $\ell$ denote two positive real numbers. Fix $V_{0}>\ell a_{0}^2$ and $S_{0}> \ell a_{0}$. We investigate the problem of maximizing the functional $a\mapsto \tilde F(a)$ with $\tilde F(a)=-ka(0)T'(0)$, where $T$ denotes the unique solution of \eqref{eq:Tbis2}, either over the set $\mathcal{V}_{a_{0},\ell,V_{0}}$ or $\mathcal{S}_{a_{0},\ell,S_{0}}$, respectively defined by \eqref{defV} and \eqref{defS}.
}
\end{quote}
We obtain the following result.

\begin{theorem}\label{thpb1D:surfbis}
Let $a_{0}$, $\beta_{0}$ and $\ell$ denote three positive real numbers, and let $V_{0}>\ell a_{0}^2$ and $S_{0}> \ell a_{0}$. 
\begin{enumerate}
\item One has
$$
\sup_{a\in \mathcal{V}_{a_{0},\ell,V_{0}}}\tilde F(a)=+\infty.
$$
\item Let us assume that $\beta\in C^0([0,\ell])$ is nonconstant, satisfies \eqref{assump:beta0} and 
\begin{equation}\label{assump:beta}
\max_{x\in [0,\ell]}\beta(x)=\beta(0).
\end{equation}
Thus, the problem of maximizing $\tilde F$ over $\mathcal{S}_{a_{0},\ell,S_{0}}$ has no solution and
$$
\sup_{a\in \mathcal{S}_{a_{0},\ell,S_{0}}}\tilde F(a)=k\pi a_{0}\int_{0}^\ell \beta(x)(\tilde T-T_{\infty})\, dx+k\pi (S-a_{0}\ell)\beta(0)(T_{d}-T_{\infty})+k\beta_{r}a_{0}^2(\tilde T(\ell)-T_{\infty}),
$$
where $\tilde T$ is the unique solution of System \eqref{eq:Tbis2} with $a(\cdot)=a_{0}$.
Moreover, every sequence $(a_{n})_{n\in \N}$ of elements of $\mathcal{S}_{a_{0},\ell,S_{0}}$ defined as in the statement of Theorem \ref{thpb1D:surf} is a maximizing sequence for this problem.
\end{enumerate}
\end{theorem}
\begin{remark}
The technique used in the proofs of Theorems \ref{thpb1D:surf} and \ref{thpb1D:surfbis} fails and cannot be easily adapted when the function $\beta$ does not satisfy \eqref{assump:beta} anymore. The issue of solving the same problem in this case is discussed and commented in Section \ref{sec:num}.
\end{remark}
\begin{proof}
We do not give all details since the proof is very similar to the ones of Theorems \ref{thpb1D:vol} and \ref{thpb1D:surf}. We only underline the slight differences in every step.
\begin{enumerate}
\item Consider the sequence $(a_{S,m})_{m\in\N^*}$ introduced in the proof of Theorem \ref{thpb1D:vol}. The same kind of computations show that
\begin{eqnarray*}
\lim_{m\to +\infty}\frac{\tilde F(a_{S,m})}{k\pi} & = & \beta(0)(T_{d}-T_{\infty})(S-a_{0}\ell)+ a_{0}\int_{0}^\ell (T(x)-T_{\infty})\, dx+\beta_{r}a_{0}^2(T(\ell)-T_{\infty})\\
& \geq & \beta(0)(T_{d}-T_{\infty})(S-a_{0}\ell),
\end{eqnarray*}
where $T$ is the unique solution of System \eqref{eq:Tbis2} with $a(\cdot)=a_{0}$. Then, we conclude similarly to the proof of Theorem \ref{thpb1D:vol}.
\item Accordingly to the three first steps of the proof of Theorem \ref{thpb1D:surf}, we first show that the aforementioned problem has no solution. We argue by contradiction, by assuming the existence of a solution $a$, introducing $b=a\sqrt{1+a'^2}$ and the perturbation 
$$
b_{\varepsilon}=b+c(\chi_{[0,\varepsilon]}-\chi_{[x_{0}-\varepsilon/2,x_{0}+\varepsilon/2]}),
$$
where $c$ is a positive constant, chosen small enough to guarantee that $b_{\varepsilon}>a_{0}$ almost everywhere in $(0,\ell)$, and $x_{0}\in (0,\ell)$ is such that $a>a_{0}$ on $(x_{0}-\varepsilon/2,x_{0}+\varepsilon/2)$. The same computations as those led in Steps 2 and 3 yield
\begin{eqnarray*}
\lim_{\varepsilon\searrow 0}\frac{\tilde F(a_{\varepsilon})-\tilde F(a)}{k\pi \varepsilon}& =& \int_{0}^\ell \beta (x)b(x)\widetilde T(x)\, dx+c\left(\beta(0)(T_{d}-T_{\infty})-\beta(x_{0})(T(x_{0})-T_{\infty})\right) \\
& & +\beta_{r} a(\ell)^2\widetilde{T}(\ell),
\end{eqnarray*}
where $\widetilde T\in H^1(0,\ell)$ is the unique function such that $\widetilde T(0)=0$ and 
\begin{eqnarray}
 \int_{0}^\ell \left(a(x)^2\widetilde{T}'(x)\varphi'(x)+\beta(x) b(x)\widetilde{T}(x)\varphi(x)\right)\, dx-\beta(x_{0}) c\varphi(x_{0})(T(x_{0})-T_{\infty})\nonumber \\
 +\beta_{r}a(\ell)^2\widetilde{T}(\ell)\varphi(\ell)=0,\label{FVTtildebis}
\end{eqnarray}
for every test function $\varphi\in H^1(0,\ell)$ satisfying $\varphi(0)=0$. An adequate choice  of test function allows to reduce the last expression to
$$
\lim_{\varepsilon\searrow 0}\frac{\tilde F(a_{\varepsilon})-\tilde F(a)}{k\pi \varepsilon} = c\frac{\beta(0)(T_{d}-T_{\infty})^2-\beta(x_{0})(T(x_{0})-T_{\infty})^2}{T_{d}-T_{\infty}}.
$$
Since $T-T_{\infty}$ is positive decreasing, the right-hand side is positive, which yields a contradiction.

It remains now to exhibit a maximizing sequence. For that purpose, we follow the lines of Step 4, by investigating the problem 
\begin{equation}\label{pbMbis}
\sup \left\{\tilde F(a), \ a\in \mathcal{S}_{a_{0},\ell,S_{0}},\  a_0\leq a\sqrt{1+a'^2}\leq M\textrm{ a.e. in }(0,\ell)\right\}
\end{equation}
and showing that a solution $\underline{a}_{M}$ satisfies necessarily \eqref{eq:bang}. Define the function $\underline{b}_{M}$ by $\underline{b}_{M}=\underline{a}_{M}\sqrt{1+\underline{a}_{M}'^2}$. As in Proposition \ref{prop:bang}, we argue by contradiction. Thus, it suffices to consider a particular perturbation $b_{\varepsilon}$ of $\underline{b}_{M}$ of the form
$$b_{\varepsilon}=\underline{b}_M+c(\chi_{[y_{0},y_{0}+\varepsilon]}-\chi_{[x_{0}-\varepsilon,x_{0}]}),
$$
with $0<y_{0}<x_{0}<\ell$. 
The proof is identical, but we need to adapt slightly the proof of Proposition \ref{prop:bang}, and in particular, to make the choices of $\varepsilon$, $x_{0}$ and $y_{0}$ precise, in order to guarantee that $\lim_{\varepsilon\searrow 0}\frac{\tilde F(a_{\varepsilon})-\tilde F(\underline{a}_M)}{k\pi \varepsilon}$ be positive, provided that $\beta$ satisfies \eqref{assump:beta}. According to this assumption, we consider $\varepsilon>0$ and $x_{0}\in (0,\ell)$ such that
$$
\left.\beta \right|_{[0,\varepsilon]}> \left.\beta \right|_{[x_{0}-\varepsilon/2,x_{0}+\varepsilon/2]}.
$$
Recall that $x_{M}=\frac{S-a_{0}\ell}{M-a_{0}}$ and notice that, for $M$ large enough, $x_{M}\in (0,\varepsilon)$. We fix $M_{0}>a_{0}$ and $y_{0}\in (0,x_{M})$ so that $\left. b\right|_{[y_{0},y_{0}+\varepsilon]}<M$ for $M>M_{0}$. Such conditions ensure that 
$$
\lim_{\varepsilon\searrow 0}\frac{\tilde F(a_{\varepsilon})-\tilde F(\underline{a}_{M})}{k\pi \varepsilon} = c\frac{\beta(y_0)(T(y_{0})-T_{\infty})^2-\beta(x_{0})(T(x_{0})-T_{\infty})^2}{T_{d}-T_{\infty}}>0,
$$
where $T$ denotes the solution of System \eqref{eq:Tbis2} with $a(\cdot)=\underline{a}_{M}$. Hence, the conclusion of Proposition \ref{prop:bang} remains true in this case.
It suffices then to mimic the rest of the proof, by considering the sequence $(\underline{b}_{M})_{M>M_{0}}$ and letting $M$ go to $+\infty$. The end of the proof is similar to the one of Theorem \ref{thpb1D:surf} and can thus be easily adapted.
\end{enumerate}
\end{proof}

\subsection{Numerical investigations}\label{sec:num}
In Proposition \ref{prop:bang}, we established an existence result for the auxiliary problem \eqref{pbM}, close to the initial one, where a pointwise upper bound constraint on the term $a\sqrt{1+a'^2}$ was added. Recall that this problem writes for a given $M>a_{0}$,
$$
\sup \left\{F(a), \ a\in \mathcal{S}_{a_{0},\ell,S_{0}},\  a_0\leq a\sqrt{1+a'^2}\leq M\textrm{ a.e. in }(0,\ell)\right\}.
$$
According to Lemma \ref{lemma:gammaCV}, solving this problem and letting $M$ go to $+\infty$ yields an approximation of the optimal value for Problem \eqref{defS}. Furthermore, this problem is also used to construct maximizing sequences in the proof of Theorem \ref{thpb1D:surf}. This is why we decided to focus the numerical investigations on the optimal design problem \eqref{pbM}.
On the following figures, its solutions are plotted for several values of the pointwise constraint $M$  in order to highlight its important role in the construction of maximizing sequences.

\begin{remark}[Brief discussion on the numerical simulations]
The simulations were obtained with a direct method applied to the optimal design problem described previously, consisting in discretizing the underlying differential equations, the optimal design $a(\cdot)$, and to reduce the shape optimization problem to some finite-dimensional maximization problem with constraints.
Equation~\eqref{eq:T} is discretized with the finite volume method, well adapted to that case because of its heat flux conservativeness property. We used two staggered grids in order to avoid the so-called {\it checkerboard phenomenon}~\cite{Allaire}. The resulting finite-dimensional optimization problem is solved by using a standard {\it interior-point method} \cite{bonnans}. We used the code \texttt{IPOPT} (see \cite{IPOPT}) combined with \texttt{AMPL} (see \cite{AMPL}) on a standard desktop machine. The resulting code works out the solution very quickly (for instance, around 3 seconds for the simulations below).
\end{remark}

We provide hereafter several numerical simulations of Problem \eqref{pbM}. Each structure has been discretized with $500$ design elements.
Figures~\ref{fig:constant} and~\ref{fig:decreasing} may be considered as a numerical illustration of the results stated in the theorems \ref{thpb1D:surf} and \ref{thpb1D:surfbis}. Indeed, by increasing the constraint bound $M$, a Dirac measure is outlined in $x = 0$, as expected for constant and decreasing profiles of $\beta(x)$ (see figure~\ref{fig:h_constant}~\ref{fig:h_decreasing}), even if different thermal behaviors due to the convective heat transfer coefficients are observed on both configuration (see the temperature profiles on the figures~\ref{fig:T_constant} and~\ref{fig:T_decreasing}).

To the contrary, the numerical results on the figures \ref{fig:increasing_step} and \ref{fig:increasing} provide some hints on physical cases that are not covered by the aforementioned theorems: on the first one, $\beta$ is chosen to be a kind of regularized {\it step function} (see Figure~\ref{fig:h_increasing_step}), whereas on the second one, it is assumed to be an increasing function (see Figure~\ref{fig:h_increasing}).

\begin{figure}[t!]
	\centering
	\subfigure{\includegraphics[height=0.35\textwidth]{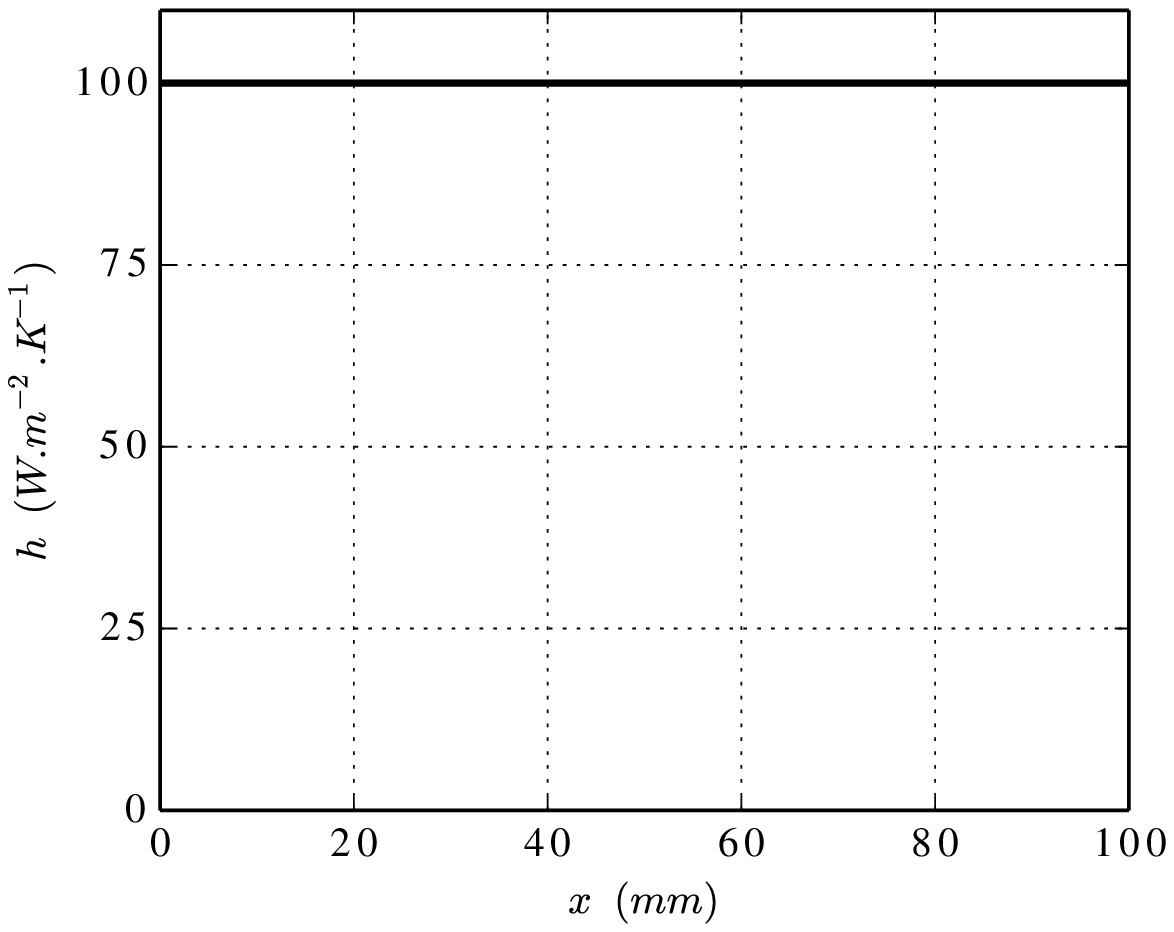}
			\label{fig:h_constant} }\quad
	\subfigure{\includegraphics[height=0.35\textwidth]{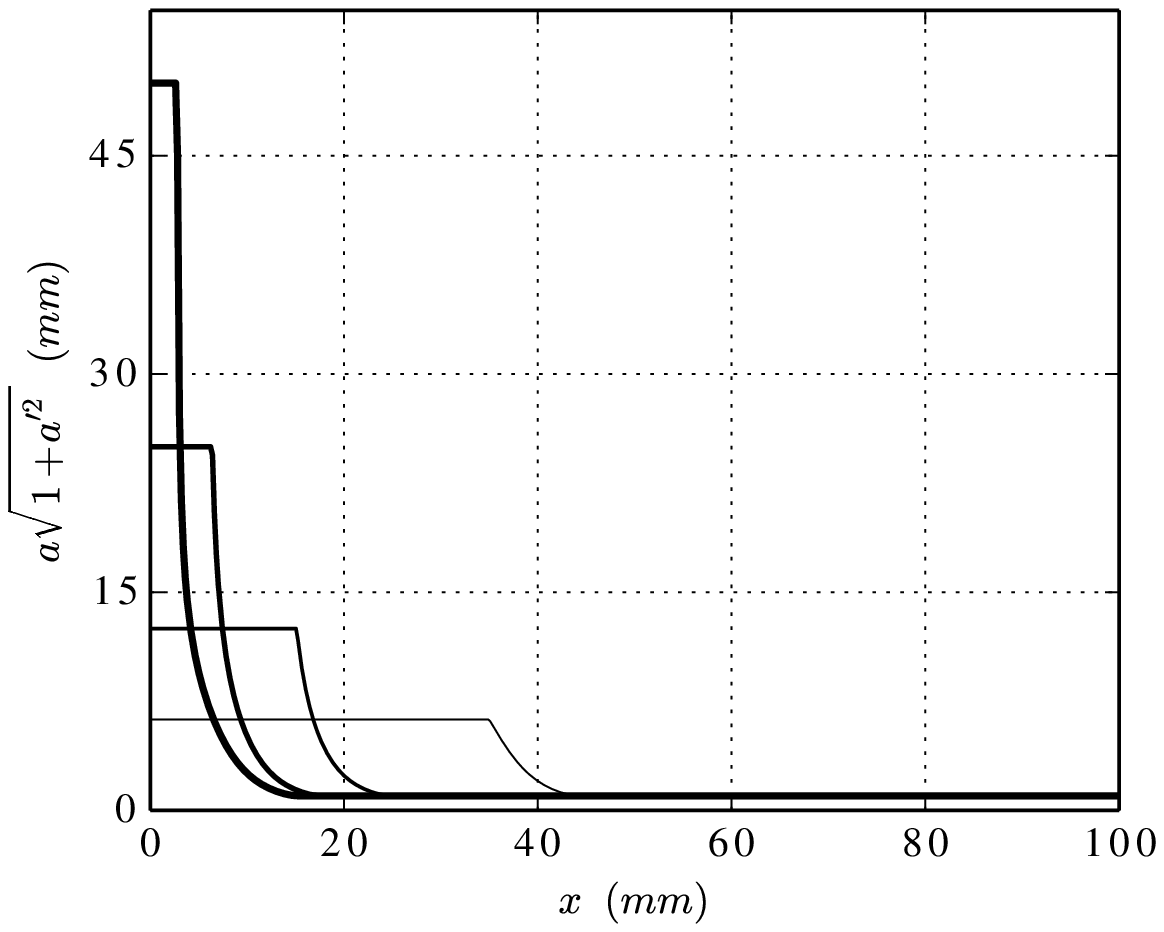}
			\label{fig:ap_constant} }\quad
	\subfigure{\includegraphics[height=0.35\textwidth]{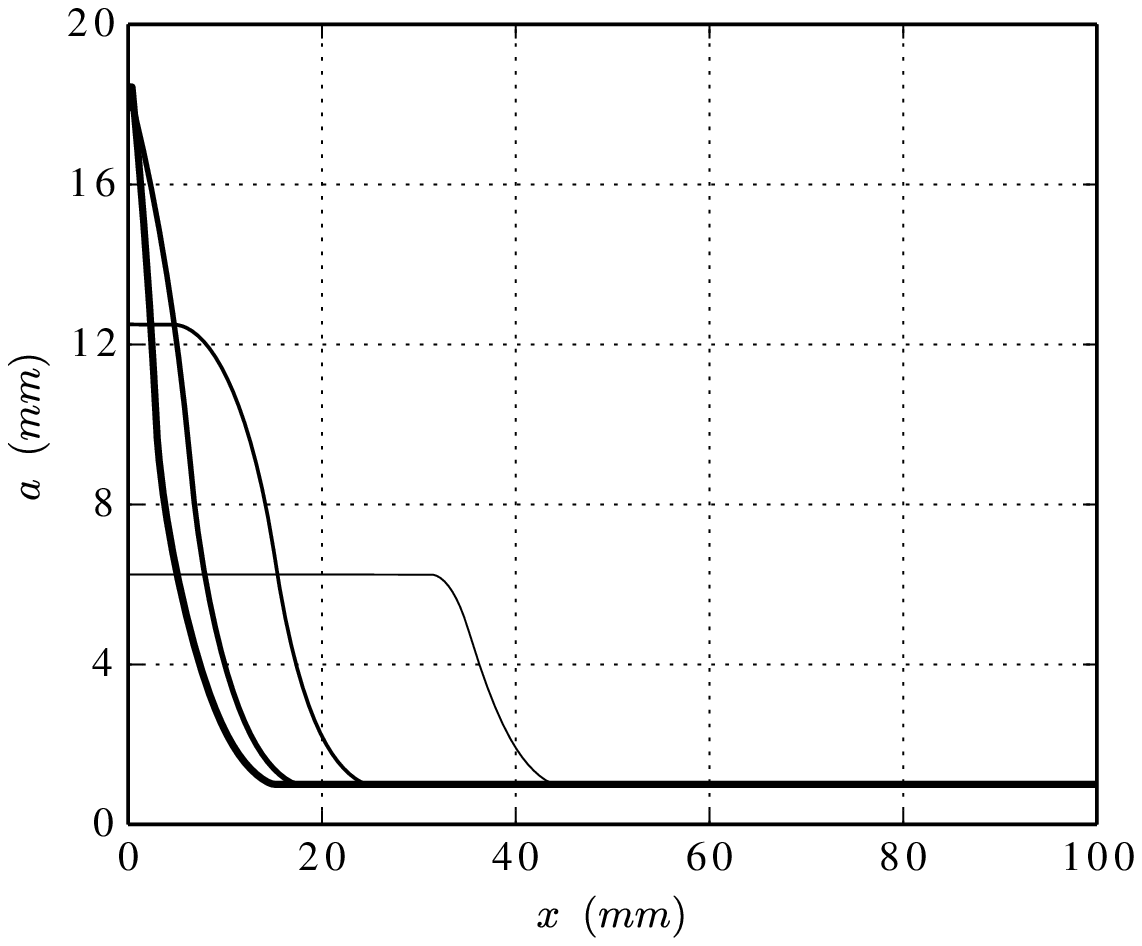}
			\label{fig:a_constant} }\quad
	\subfigure{\includegraphics[height=0.35\textwidth]{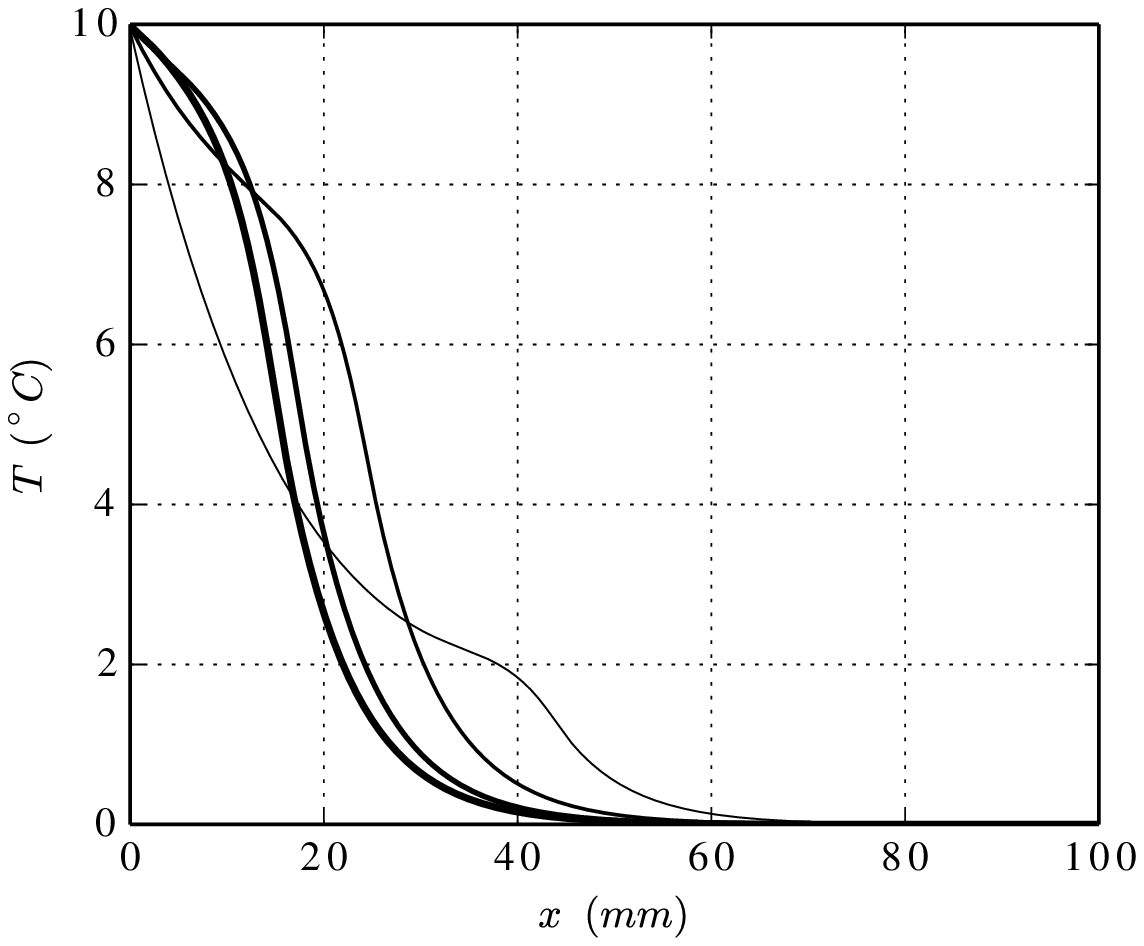}
			\label{fig:T_constant} }\quad
	\caption{Numerical solutions for a constant profile $h$ displayed in Fig.~\ref{fig:h_constant}, with $a_0 = 1 \textrm{ mm}$, $\ell = 100 \textrm{ mm}$, $S = 6 \pi a_0 \ell$, $T_d = 10\, ^{\circ}C$, $T_{\infty} = 0\, ^{\circ}C$, $h_r = h(\ell)$ and $k=10 \,  W.m^{-2}.K^{-1}$. The constraint $M$ is progressively untightened from $M=6.25 \textrm{ mm}$ (\protect\rule[0.25em]{3mm}{.5pt}), to $M=12.5 \textrm{ mm}$ (\protect\rule[0.25em]{3mm}{.75pt}), then $M=25 \textrm{ mm}$ (\protect\rule[0.25em]{3mm}{1pt}) and finally up to $M=50 \textrm{ mm}$ (\protect\rule[0.25em]{3mm}{1.25pt}). Top left: convective coefficient $h$; top right: function $a\sqrt{1+a'^2}$; bottom left: function $a$; bottom right: temperature $T$.
	\label{fig:constant}}
\end{figure}

\begin{figure}[b!]
	\centering
	\subfigure{\includegraphics[height=0.35\textwidth]{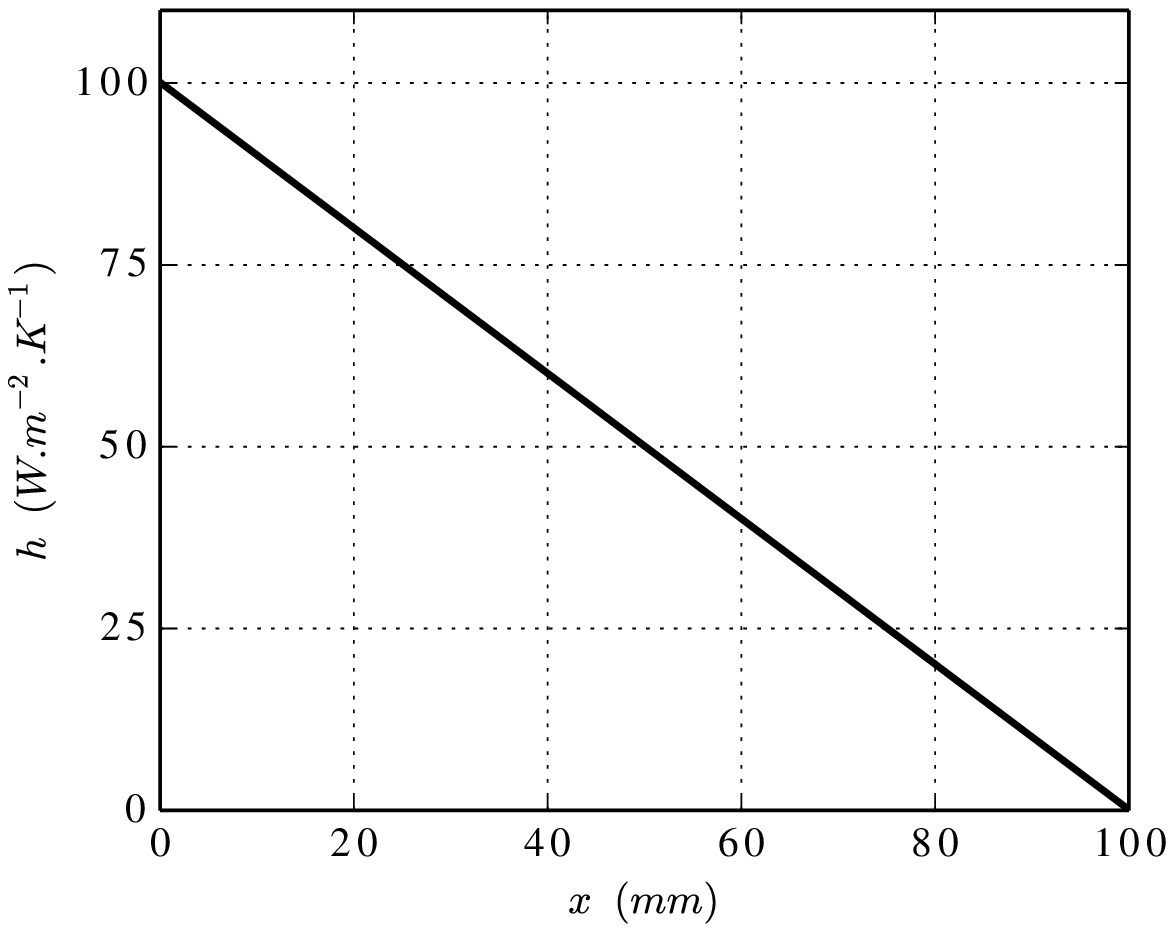}
			\label{fig:h_decreasing} }\quad
	\subfigure{\includegraphics[height=0.35\textwidth]{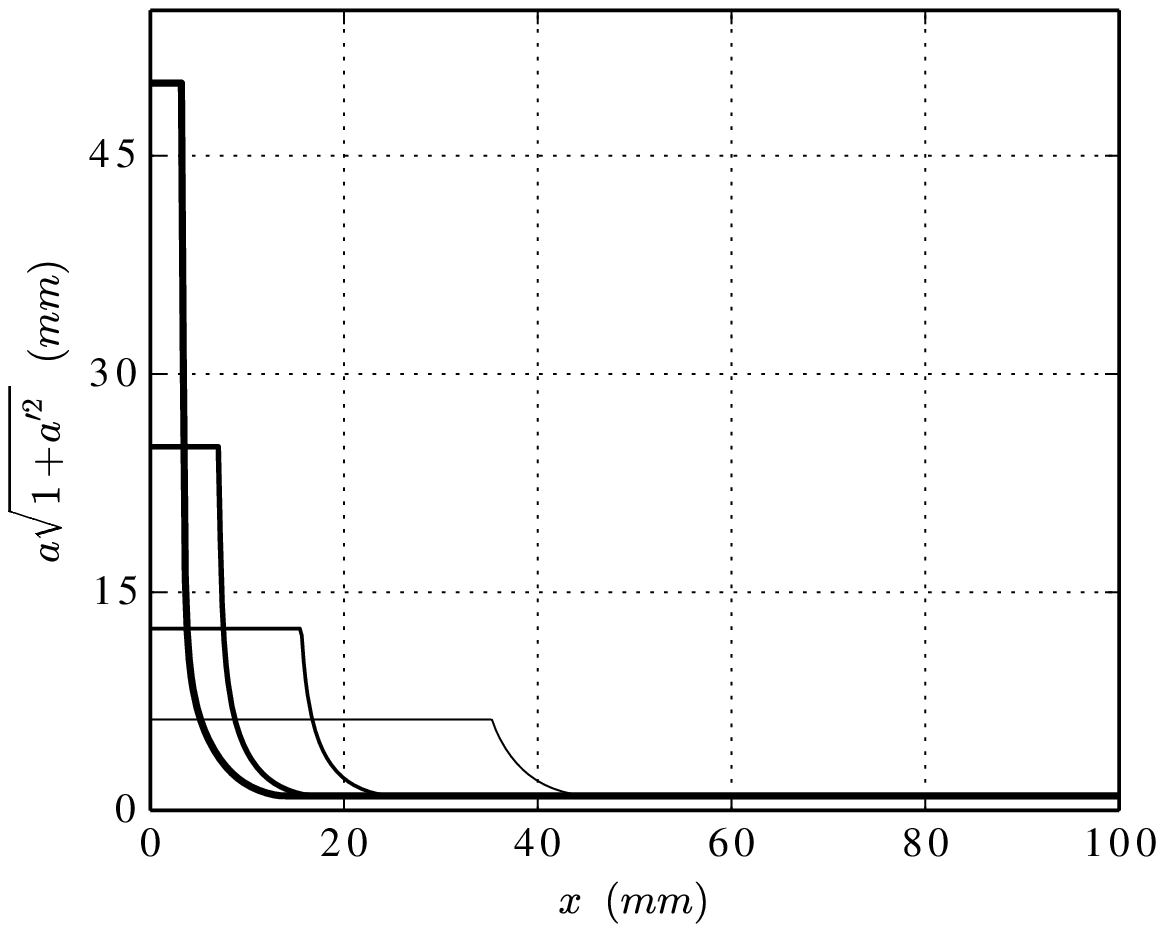}
			\label{fig:ap_decreasing} }\quad
	\subfigure{\includegraphics[height=0.35\textwidth]{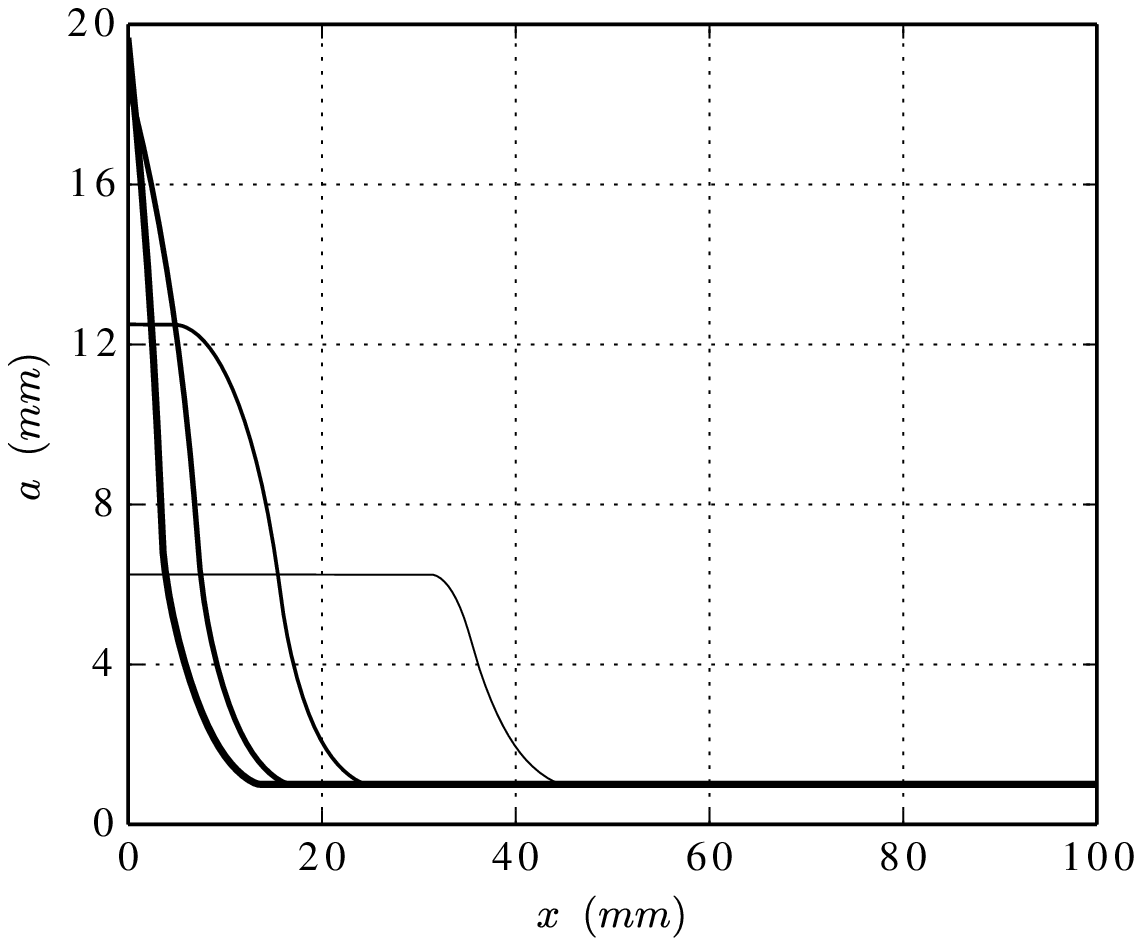}
			\label{fig:a_decreasing} }\quad
	\subfigure{\includegraphics[height=0.35\textwidth]{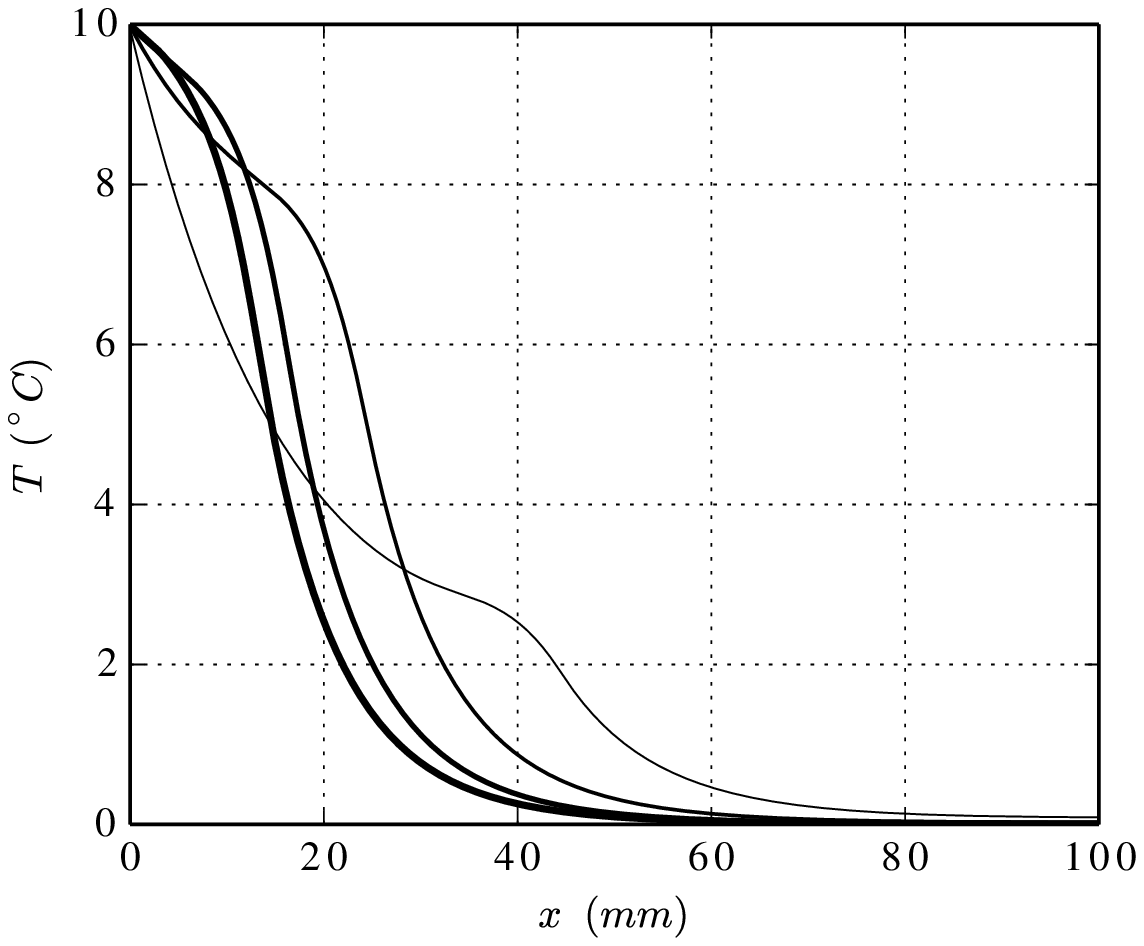}
			\label{fig:T_decreasing} }\quad
	\caption{Numerical solutions for a decreasing profile $h$ displayed in Fig.~\ref{fig:h_decreasing}, with $a_0 = 1 \textrm{ mm}$, $\ell = 100 \textrm{ mm}$, $S = 6 \pi a_0 \ell$, $T_d = 10\, ^{\circ}C$, $T_{\infty} = 0\, ^{\circ}C$, $h_r = h(\ell)$ and $k=10 \,  W.m^{-2}.K^{-1}$. The constraint $M$ is progressively untightened from $M=6.25 \textrm{ mm}$ (\protect\rule[0.25em]{3mm}{.5pt}), to $M=12.5 \textrm{ mm}$ (\protect\rule[0.25em]{3mm}{.75pt}), then $M=25 \textrm{ mm}$ (\protect\rule[0.25em]{3mm}{1pt}) and finally up to $M=50 \textrm{ mm}$ (\protect\rule[0.25em]{3mm}{1.25pt}). Top left: convective coefficient $h$; top right: function $a\sqrt{1+a'^2}$; bottom left: function $a$; bottom right: temperature $T$.
	\label{fig:decreasing}}
\end{figure}

\begin{figure}[t!]
	\centering
	\subfigure{\includegraphics[height=0.35\textwidth]{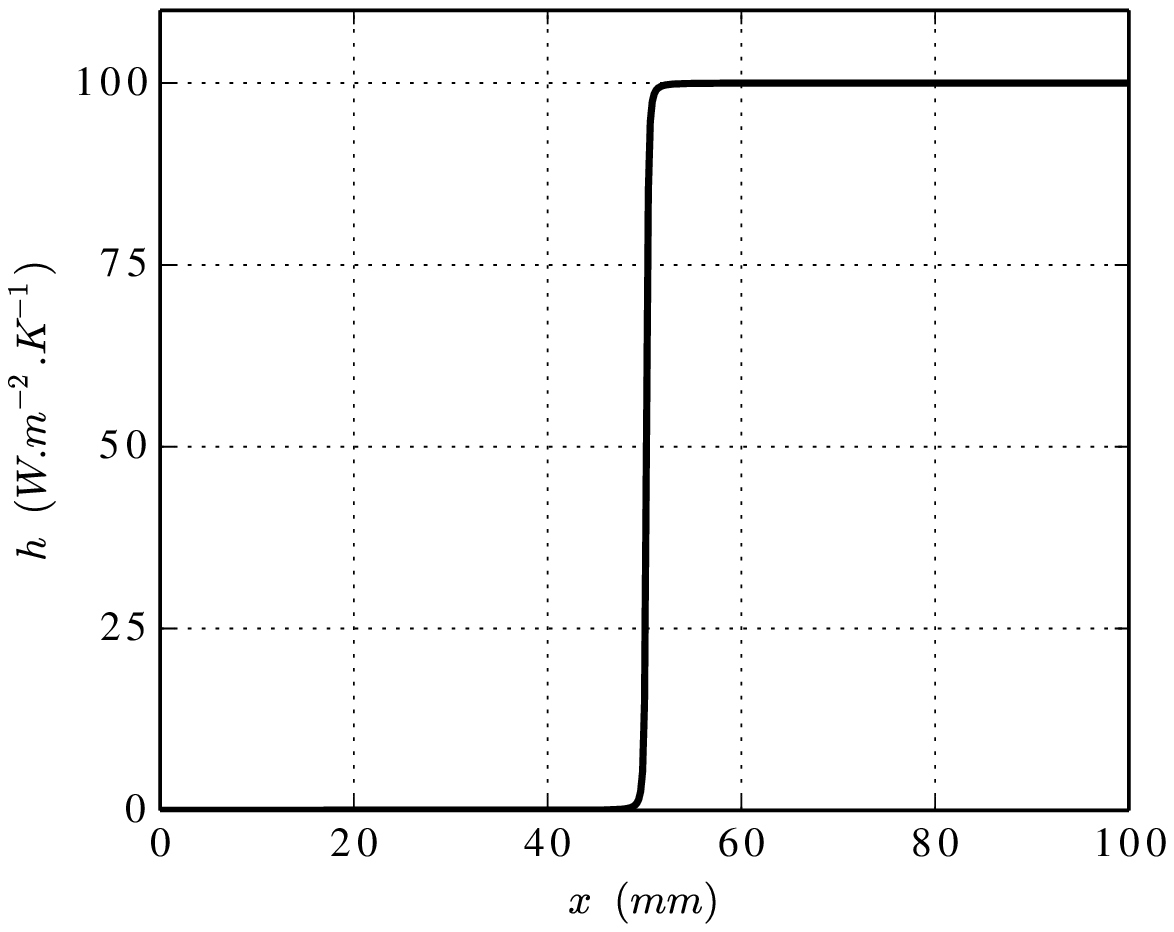}
			\label{fig:h_increasing_step} }\quad
	\subfigure{\includegraphics[height=0.35\textwidth]{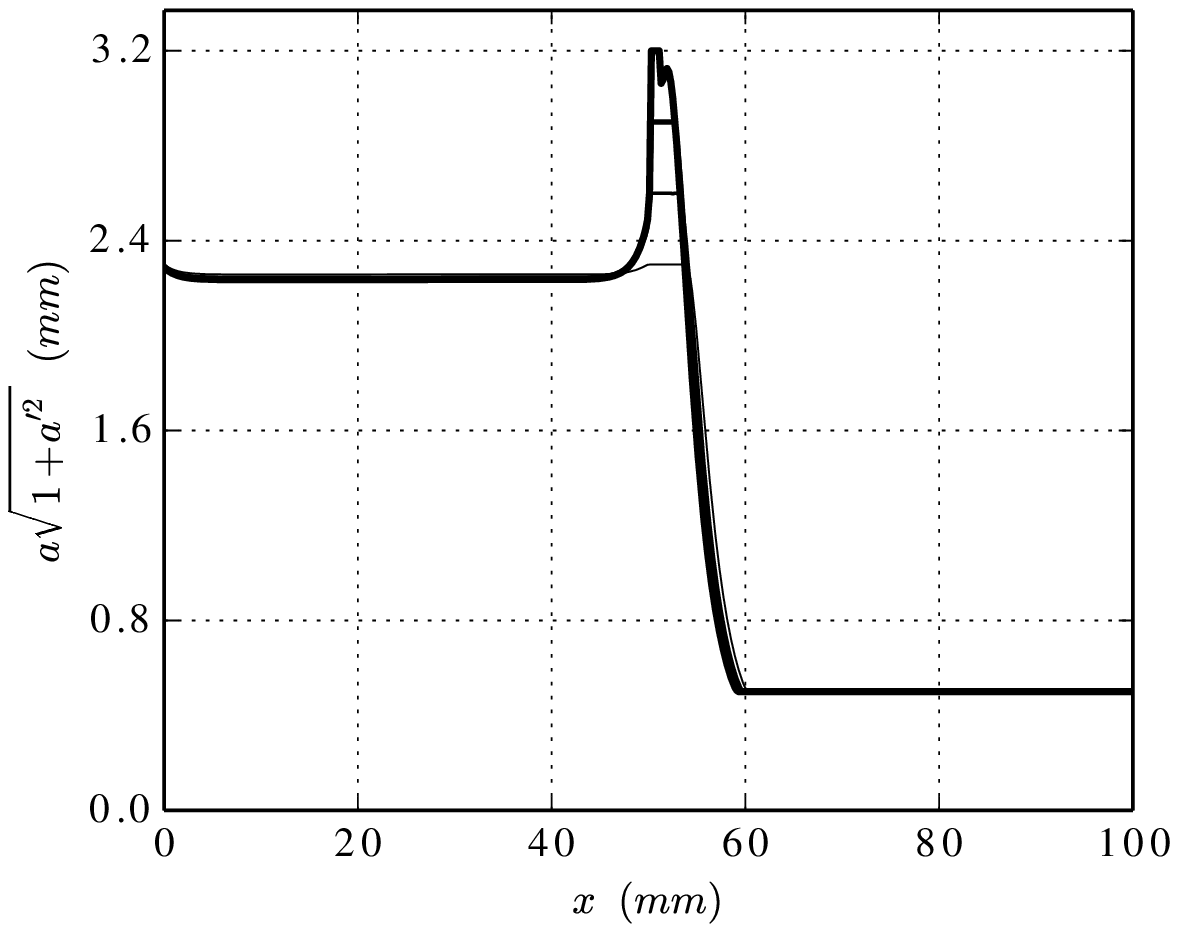}
			\label{fig:ap_increasing_step} }\quad
	\subfigure{\includegraphics[height=0.35\textwidth]{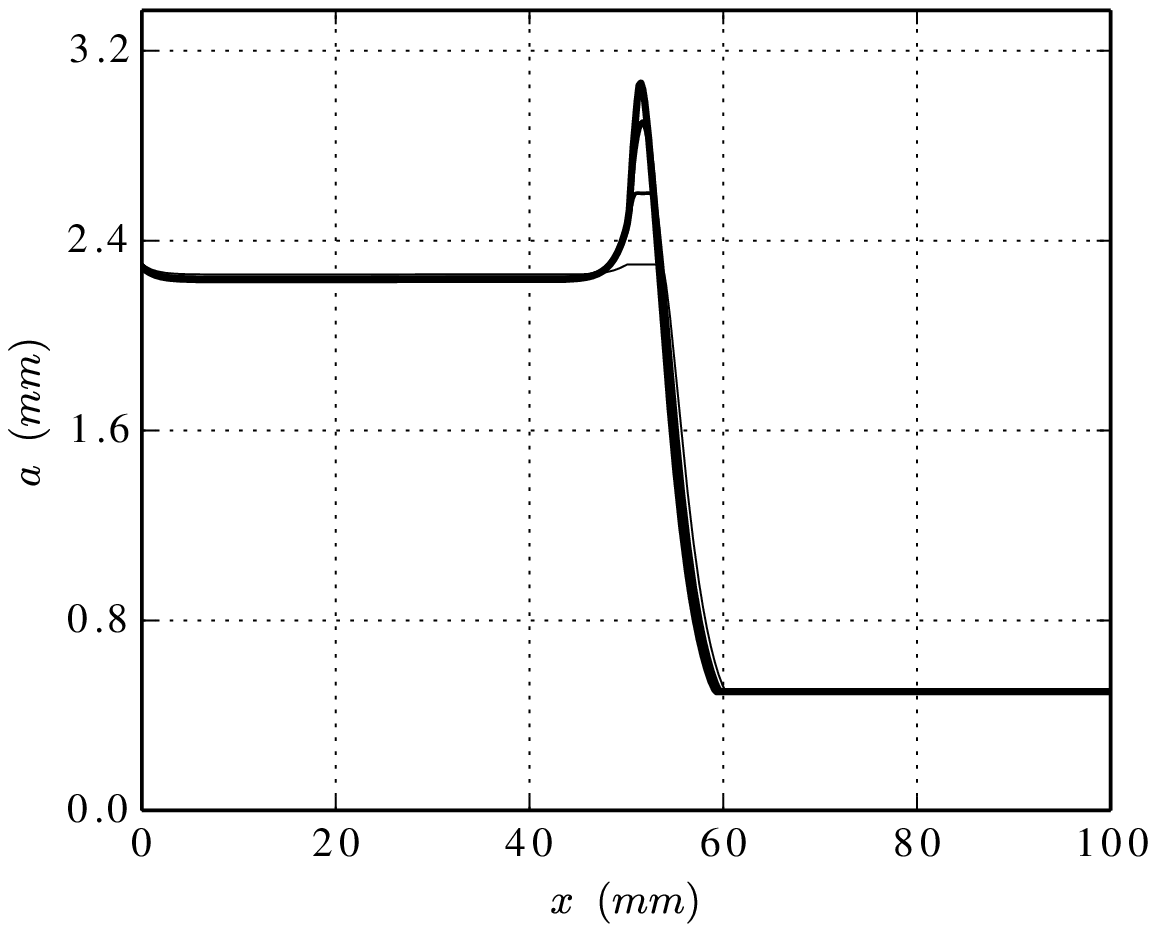}
			\label{fig:a_increasing_step} }\quad
	\subfigure{\includegraphics[height=0.35\textwidth]{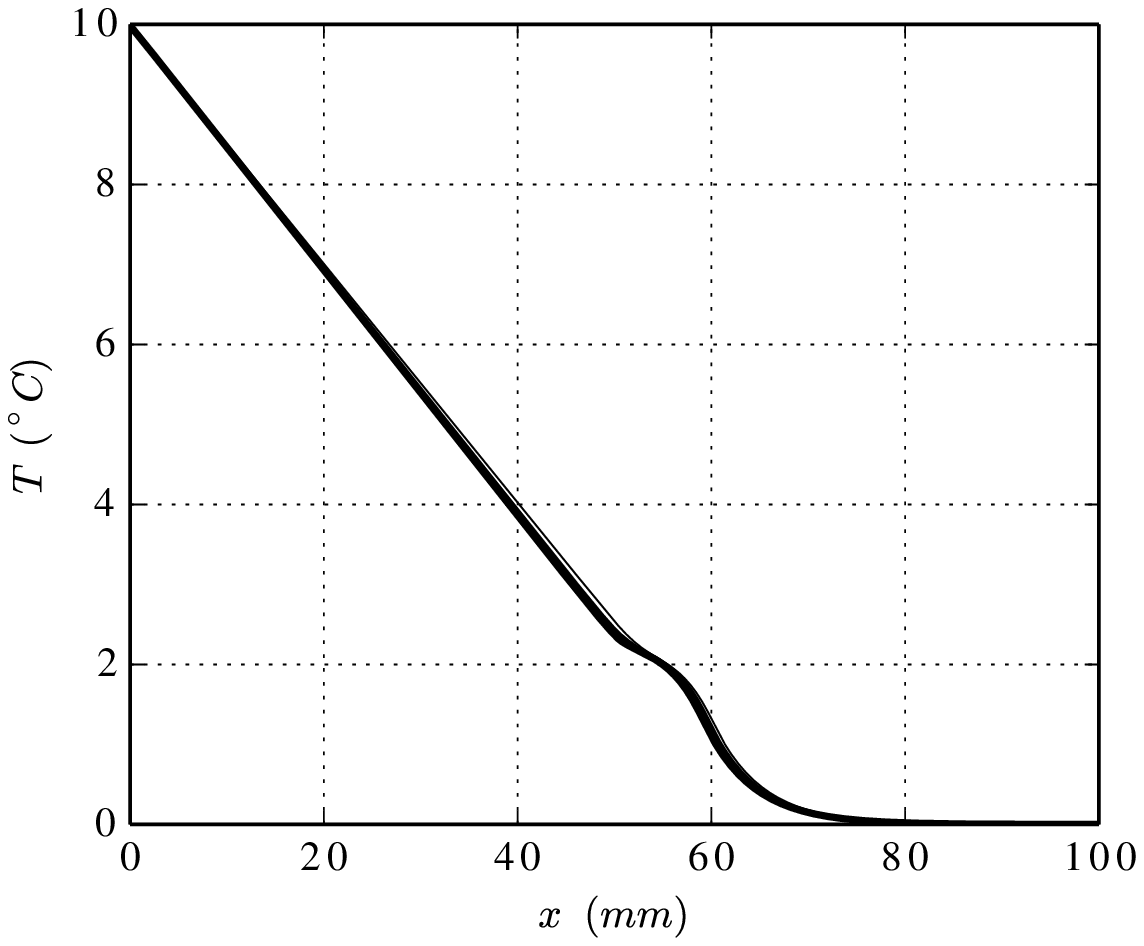}
			\label{fig:T_increasing_step} }\quad
	\caption{Numerical solutions for the increasing profile $h$ with a step taking place at $x_s=50 \textrm{ mm}$ and with $a_0 = 0.5 \textrm{ mm}$, $\ell = 100 \textrm{ mm}$, $S = 6 \pi a_0 \ell$, $T_d = 10\, ^{\circ}C$, $T_{\infty} = 0\, ^{\circ}C$, $h_r = h(\ell)$ and $k=10 \,  W.m^{-2}.K^{-1}$. The constraint $M$ is progressively untightened from $M=2.3 \textrm{ mm}$ (\protect\rule[0.25em]{3mm}{.5pt}), to $M=2.6 \textrm{ mm}$ (\protect\rule[0.25em]{3mm}{.75pt}), then $M=2.9 \textrm{ mm}$ (\protect\rule[0.25em]{3mm}{1pt}) and finally up to $M=3.2 \textrm{ mm}$ (\protect\rule[0.25em]{3mm}{1.25pt}). Top left: convective coefficient $h$; top right: function $a\sqrt{1+a'^2}$; bottom left: function $a$; bottom right: temperature $T$.
	\label{fig:increasing_step}}
\end{figure}

\begin{figure}[t!]
	\centering
	\subfigure{\includegraphics[height=0.35\textwidth]{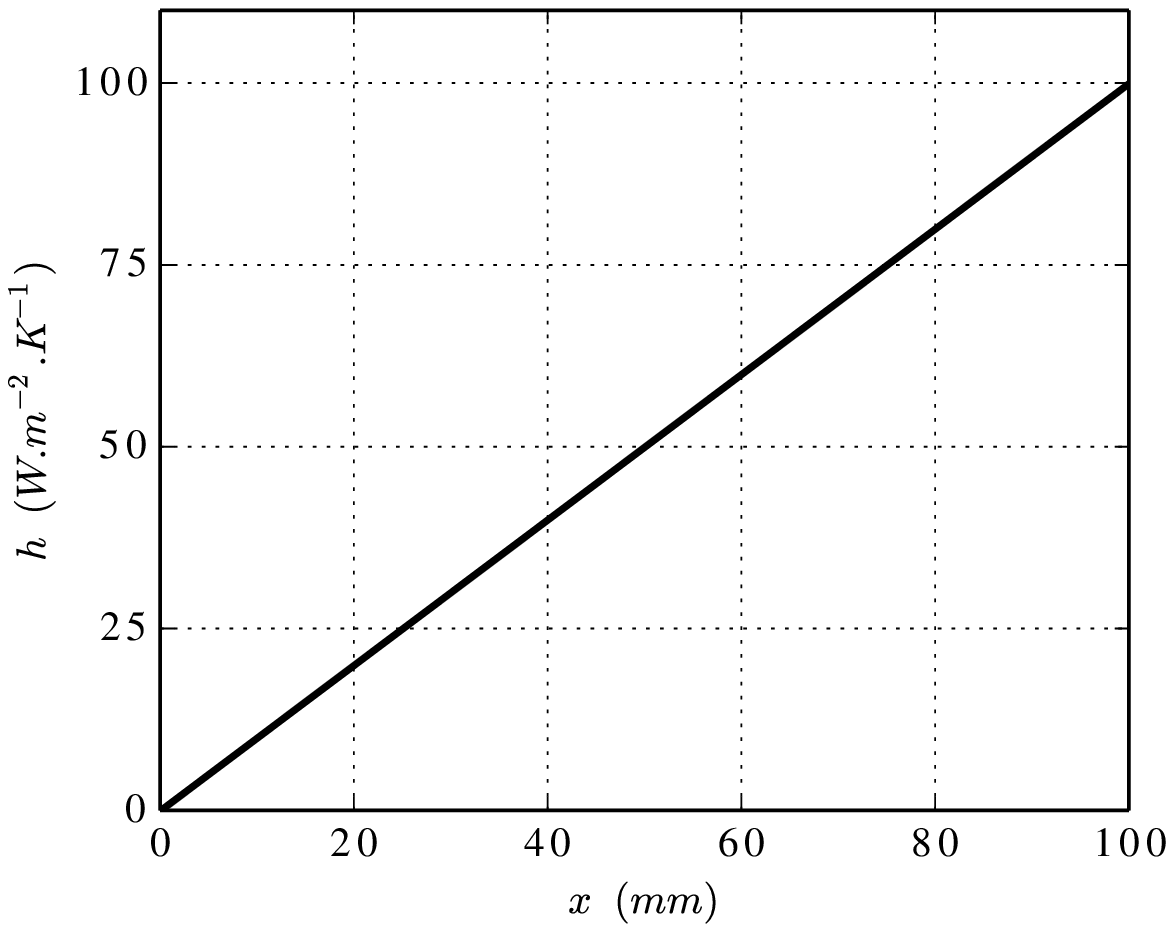}
			\label{fig:h_increasing} }\quad
	\subfigure{\includegraphics[height=0.35\textwidth]{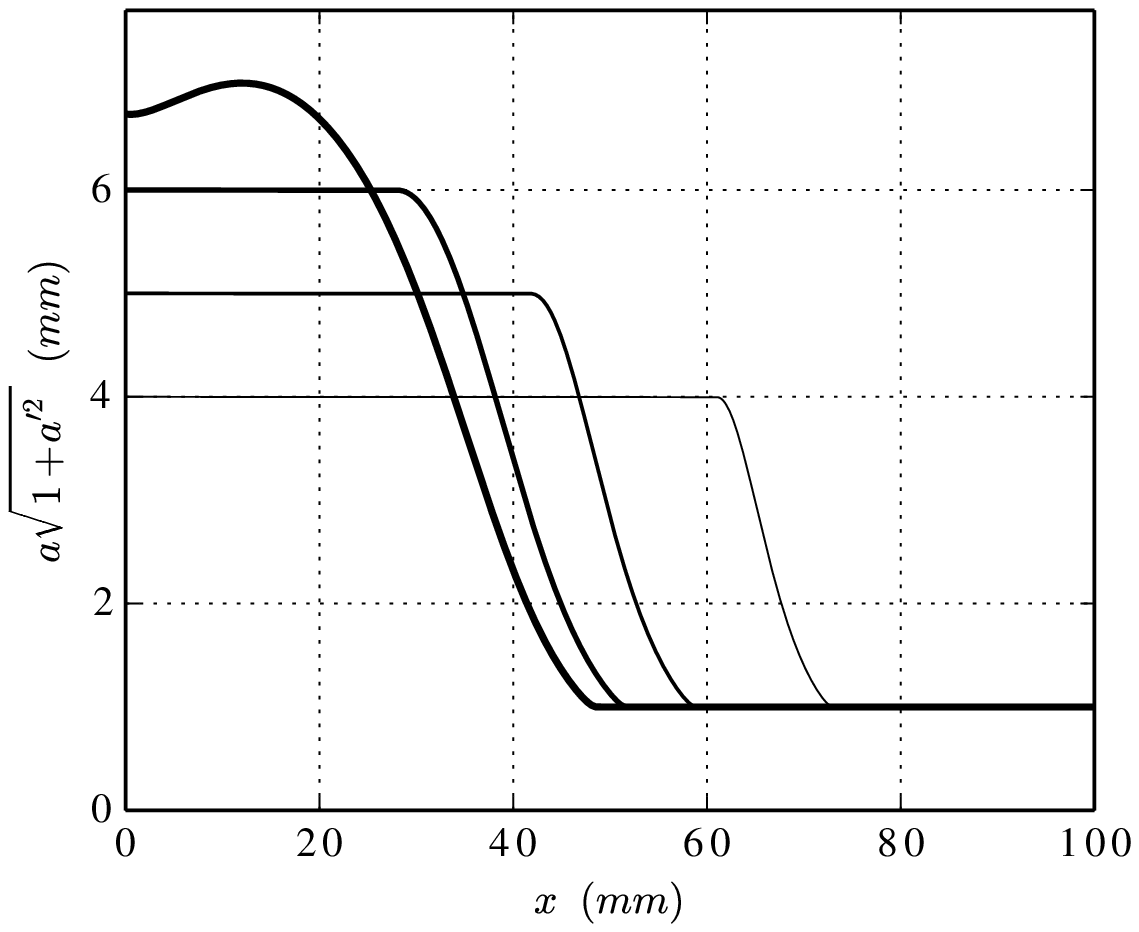}
			\label{fig:ap_increasing} }\quad
	\subfigure{\includegraphics[height=0.35\textwidth]{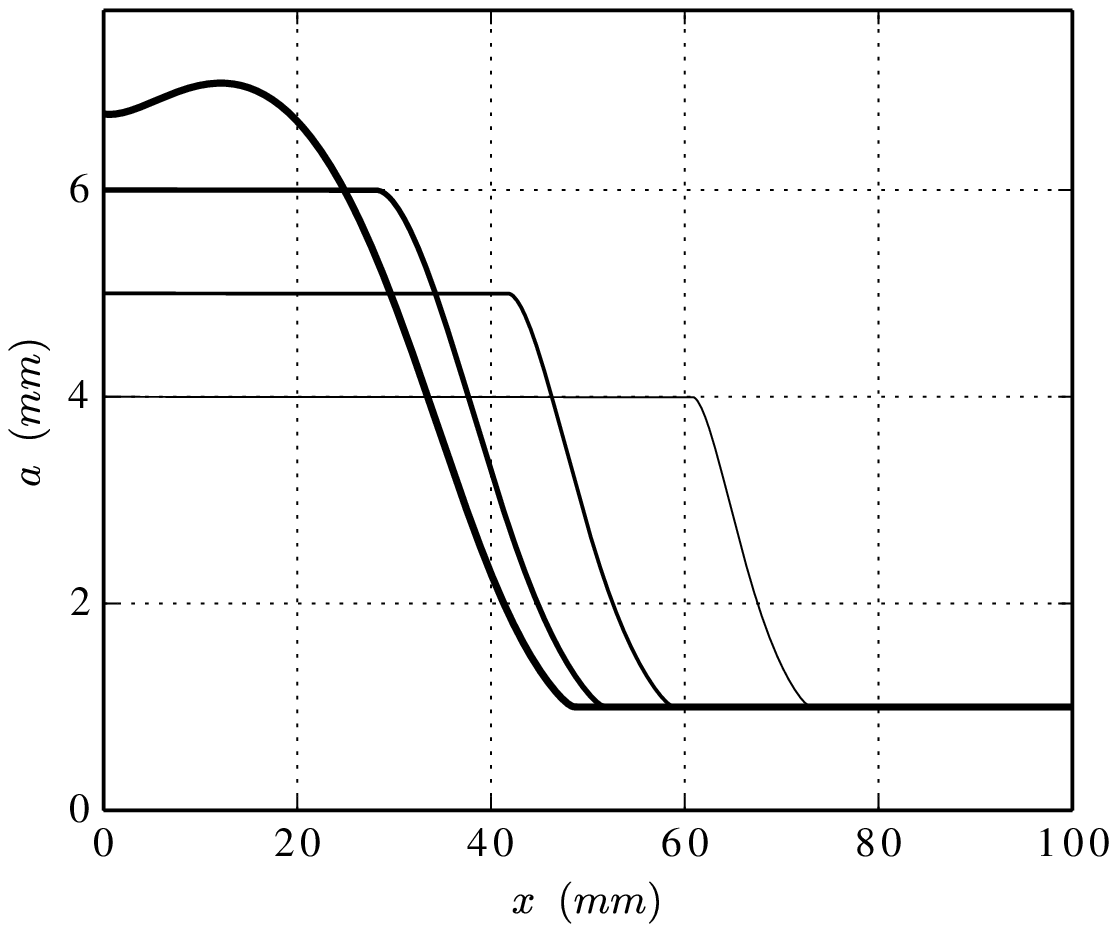}
			\label{fig:a_increasing} }\quad
	\subfigure{\includegraphics[height=0.35\textwidth]{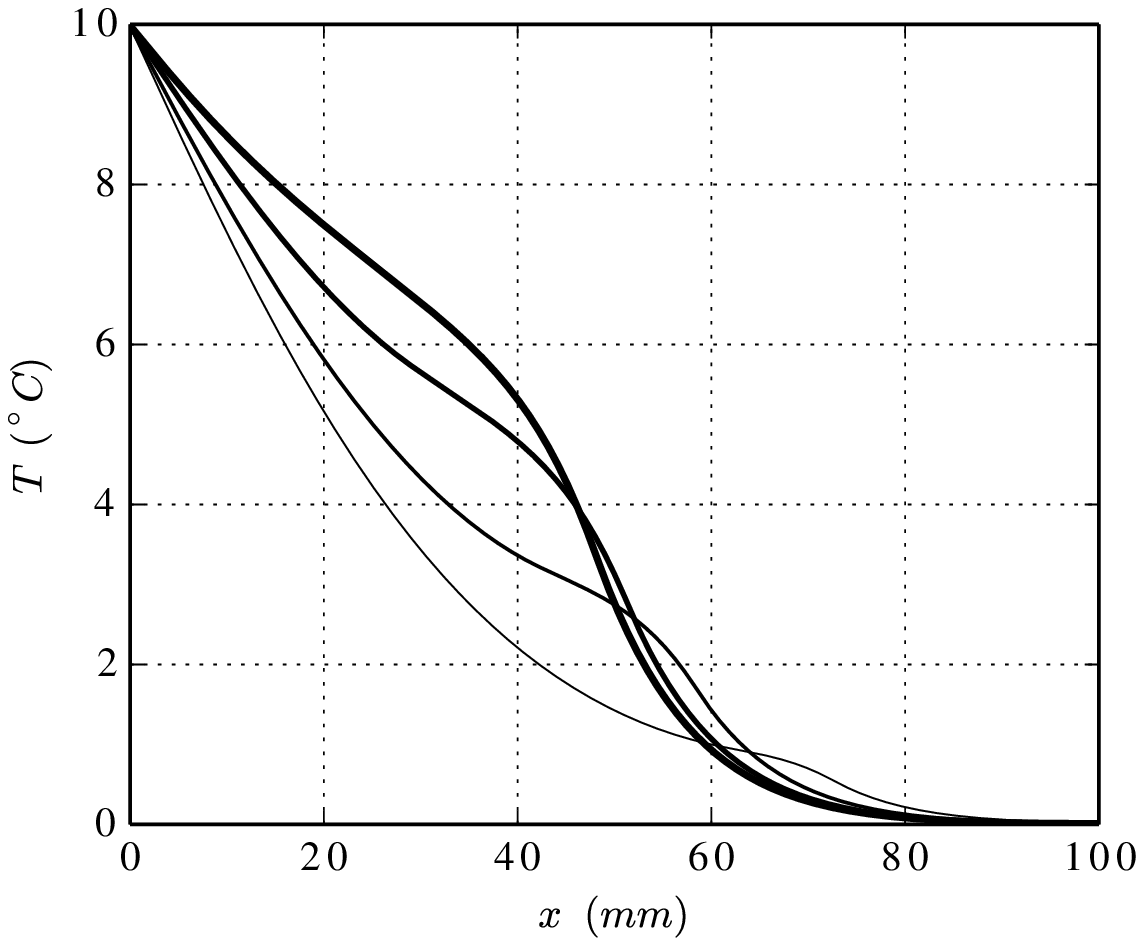}
			\label{fig:T_increasing} }\quad
	\caption{Numerical solutions for the increasing profile $h$ displayed in Fig.~\ref{fig:h_increasing} and with $a_0 = 1 \textrm{ mm}$, $\ell = 100 \textrm{ mm}$, $S = 6 \pi a_0 \ell$, $T_d = 10\, ^{\circ}C$, $T_{\infty} = 0\, ^{\circ}C$, $h_r = h(\ell)$ and $k=10 \,  W.m^{-2}.K^{-1}$. The constraint $M$ is progressively untightened from $M=4 \textrm{ mm}$ (\protect\rule[0.25em]{3mm}{.5pt}), to $M=5 \textrm{ mm}$ (\protect\rule[0.25em]{3mm}{.75pt}) and then $M=6 \textrm{ mm}$ (\protect\rule[0.25em]{3mm}{1pt}). The last solution (\protect\rule[0.25em]{3mm}{1.25pt}) is computed by fully removing the constraint $a \sqrt{ 1 + a^{\prime 2} } < M$. Top left: convective coefficient $h$; top right: function $a\sqrt{1+a'^2}$; bottom left: function $a$; bottom right: temperature $T$.
	\label{fig:increasing}}
\end{figure}

The numerical simulations on the cases that are not covered by our theoretical study suggest that several situations may arise:
\begin{itemize}
	\item according to Figure \ref{fig:increasing_step}, one could expect a nonexistence result, since the term $\underline{a}_{M}\sqrt{1+\underline{a}_{M}'^2}$ seems to converge in the sense of measures to the sum of a regular function and a Dirac measure at the point where the step of the function $\beta$ occurs ($x_s=50 \textrm{ mm}$), as $M$ tends to $+\infty$.
	\item at the opposite, one could maybe expect that the optimal design problem corresponding to the profile plotted on Figure \ref{fig:increasing}, where $\beta$ is an increasing affine function, has a solution. As a matter of fact, even if the pointwise constraint $a\sqrt{1+a'^2}$ is removed, the optimal design seems to converge towards a regular design function, without Dirac measures (the bolder profiles on figures~\ref{fig:ap_increasing} and~\ref{fig:a_increasing}). 
\end{itemize}

Finally, let us comment on the interesting mathematical issue of investigating the cases where Assumption \eqref{assump:beta} is not satisfied anymore. In that case, the perturbation $b_{\varepsilon}$ introduced in the proof of Theorem \ref{thpb1D:surfbis} (and even its general version used in Proposition \ref{prop:bang}) does not permit to conclude to the nonexistence of solutions, and would probably need a specific study. And yet, it is not clear whether the related optimal design problem has a solution, or not.

\subsection{Conclusion and perspectives}\label{sec:ccl_pers}
In this article, we addressed the issue of finding the optimal shape of a fin, by assuming that its shape is axisymmetric and considering as physical model of the temperature along its axis a simplified one-dimensional Sturm-Liouville system, much used in the engineering literature~\cite{Bergman,Welty}. Two natural constraints for the shape optimization problem have been investigated, by imposing a maximal bound either on the volume or the lateral surface of the fin. In both cases, we proved in the theorems \ref{thpb1D:vol}, \ref{thpb1D:surf} and \ref{thpb1D:surfbis} that the optimal design problems \eqref{defV} and \eqref{defS} have no solution, and we have exhibited maximizing sequences. More precisely, we showed that there is no optimal shape in the set of regular radii $a$, but that a nearly optimal shape is given by the function $a_{S,M}$ defined by \eqref{def:aSm} and displayed on Figure \ref{Fig:suitemax1}: it is highly oscillating in a neighborhood of $x=0$, and then flat on the rest of the interval. 
For such radii, the temperature inside the fin might not be independent of the polar variable $r$ anymore: this hypothesis, from which we derived the model, becomes quite questionable. 
 
It would thus be natural to investigate a more elaborated three dimensional model taking into account the dependence of the temperature $T$ with respect to each space variable. Consider a fin represented by a simply connected and bounded domain $\Omega$, and introduce $\Gamma_{i}$, the inlet of the fin, $\Gamma_{lat}$ its lateral surface and $\Gamma_{o}$ the outlet of the fin,
so that $\partial\Omega=\Gamma_{i}\cup \Gamma_{lat}\cup \Gamma_{o}$.
A possible temperature model for this fin writes
\begin{equation}\label{eq:3Dcyl}
	\begin{array}{ll}
		\triangle T = 0 &  \textrm{in }\Omega \\
		-k \frac{\partial T}{\partial \nu} = h (T - T_{\infty}) & \textrm{on } \Gamma_{lat} \\
		T = T_{d} & \textrm{on } \Gamma_{i} \\
-k \frac{\partial T}{\partial \nu} = h_r(T - T_{\infty}) & \textrm{on } \Gamma_{o} 
	\end{array}
\end{equation}
where $\frac{\partial}{\partial \nu}$ is the outward normal derivative on the boundary $\partial\Omega$.

As previously, it is convenient to maximize the heat flux at the inlet, given by
\begin{equation}
F(\Omega)=- k\int_{\Gamma_{i}} \frac{\partial T}{ \partial \nu}\, d\sigma,
\end{equation} 
over an admissible class of domains, typically those domains $\Omega$ whose inlet $\Gamma_{i}$ and outlet $\Gamma_{o}$ are fixed, and whose volume or lateral surface is prescribed.

In particular, one of the challenging aspects of this problem lies in the fact that the techniques used within this article cannot be directly used to extend our results to such kind of fin models.

\end{document}